\documentclass[a4paper,11pt]{amsart}
\usepackage{hyperref}
\usepackage{soul}

\usepackage[usenames,dvipsnames,svgnames,table]{xcolor}
\usepackage[all]{xy}
\usepackage{enumitem}
\usepackage[normalem]{ulem}
\newcommand{\tsk}[1]{\textcolor{YellowOrange}}

\usepackage{tikz}
\usepackage{tikz-cd}
\usepackage{graphicx}
\usepackage{pgf,tikz}
\usetikzlibrary{arrows}

\makeatletter
\def\@endtheorem{\endtrivlist}% NEW
\makeatother
\usepackage[active]{srcltx}
\usepackage{amsmath}
\usepackage{amssymb}
\usepackage{amscd}
\usepackage{amsthm}
\usepackage{mathrsfs}

\usepackage{nicefrac}

\usepackage{cancel}

\newtheorem{teo}{Theorem}[section]
\newtheorem{defin}[teo]{Definition}
\newtheorem{prop}[teo]{Proposition}
\newtheorem{cor}[teo]{Corollary}
\newtheorem{lemma}[teo]{Lemma}

\theoremstyle{definition}

\newtheorem{remark}[teo]{Remark}

\newtheorem{claim}{Claim}[section]
\newtheoremstyle{dico}% name of the style to be used
 {\baselineskip}   % ABOVESPACE
  {\topsep}   % BELOWSPACE
  {}  % BODYFONT
  {0pt}       % INDENT (empty value is the same as 0pt)
  {} % HEADFONT
  {.}         % HEADPUNCT
  {5pt plus 1pt minus 1pt} % HEADSPACE
  {}          % CUSTOM-HEAD-SPEC
\theoremstyle{dico}

\numberwithin{equation}{section}

\newcommand{\ra}{\rightarrow}

\newcommand{\meno}{^{-1}}

\newcommand{\comp}{\circ}

\renewcommand{\setminus}{-}

\newcommand{\om}{\omega}

\renewcommand{\phi}{\varphi}

\newcommand{\lra}{\longrightarrow}

\newcommand{\PP}{\mathbb{P}^1}

\newcommand{\OO}{\mathcal{O}}

\newcommand{\mihi}[1]{}

\newcommand{\Nm}{\operatorname{Nm}}
\newcommand{\cov}{\pi: D\rightarrow C}

\begin{document}
 
\pagestyle{myheadings}

\title{The fibres of the ramified Prym map}

\author{P. Frediani, J.C. Naranjo and I. Spelta}

\address{Paola Frediani  \\ Universit\`a degli Studi di Pavia  \\ Dipartimento di Matematica \\ Via Ferrata 5  \\ 27100 Pavia, Italy  }
\email{paola.frediani@unipv.it}

\address{Juan Carlos Naranjo \\ Departament de Matem\`atiques i Inform\`atica \\ Universitat de Barcelona \\Spain }
\email{jcnaranjo@ub.edu}

\address{Irene Spelta  \\ Universit\`a degli Studi di Pavia  \\ Dipartimento di Matematica \\ Via Ferrata 5  \\ 27100 Pavia, Italy  }
\email{irene.spelta01@universitadipavia.it}

\thanks{The first and third authors were partially supported by MIUR PRIN 2017
	``Moduli spaces and Lie Theory'' ,  by MIUR, Programma Dipartimenti di Eccellenza
	(2018-2022) - Dipartimento di Matematica ``F. Casorati'',
	Universit\`a degli Studi di Pavia and by INdAM (GNSAGA).  
	The second author was partially supported by the Proyecto de Investigaci\'on MTM2015-65361-P}
\date{}

	\begin{abstract}
		We study the ramified Prym map $\mathcal P_{g,r} \longrightarrow \mathcal A_{g-1+\frac r2}^{\delta}$ which assigns to a ramified double cover of a smooth irreducible curve of genus $g$ ramified in $r$ points the Prym variety of the covering. We focus on the six cases where the dimension of the source is strictly greater than the dimension of the target
		giving a geometric description of the generic fibre. We also give an explicit example of a totally geodesic curve which is an irreducible component of a fibre of the Prym map ${\mathcal P}_{1,2}$.  
	\end{abstract}
	\maketitle
	\setcounter{section}{-1}
	\section{Introduction}
	The Prym map $\mathcal P_{g,r}$ assigns to a degree $2$ morphism  $\pi: D \longrightarrow C$ of smooth complex irreducible curves ramified in an even number of points $r\ge 0$, a polarized abelian variety $P(\pi)=P(D,C)$ of dimension $g - 1 + \frac r2 $, where $g$ is the genus of $C$. We assume $g>0$ throughout  the paper.  
	The variety $P(\pi)$ is called the Prym variety of $\pi$ and is defined as the connected component of the origin of the kernel of the norm map $\Nm_{\pi }:JD \lra JC.$ 
	Hence, denoting by $\mathcal R_{g,r} $ the moduli space of isomorphism classes of the morphisms $\pi$, we have maps:
	\begin{equation*}
		\mathcal P_{g,r} : \mathcal R_{g,r} \longrightarrow \mathcal A^\delta_{g-1+\frac{r}{2}},
	\end{equation*}
	to the moduli space of abelian varieties of dimension  $g-1+\frac{r}{2}$ with polarization type
	$\delta:=(1,\ldots, 1,2, \ldots ,2)$, with $2$  repeated $g$ times if $r>0$ and $g-1$ times if $r=0$.
	
	The case $r=0$ is very classical. Indeed, Prym varieties of unramified coverings are principally polarized abelian varieties and they have been studied for over one hundred years, initially by Wirtinger, Schottky and Jung (among others) in the second half of the $19$th century from the analytic point of view.  They were studied later from an algebraic point of view in the seminal work of Mumford \cite{mumford} in 1974. We refer to \cite[section 1]{Farkas} for a historical account. Since Mumford's work, a lot of information has been obtained about the unramified (or ``classical'') Prym map $\mathcal P_{g,0}$. This theory is strongly related with the study of the Jacobian locus, Schottky equations and rationality problems among other topics.  It is known that $\mathcal P_{g,0}$ is generically injective for $g\ge 7$ but never injective (see  \cite{donagi} and the references therein). Moreover, in low genus, a detailed study of the structure of the fibre was provided by the works of Verra (\cite{verra}, for $g=3$),  Recillas  (\cite{ReciTrig} for $g=4$), Donagi (\cite{donagi} for $g=5$) and Donagi and Smith (\cite{ds} for $g=6$). All these results have been summarized under a uniform presentation in the fundamental work of Donagi \cite{donagi}. As we explain below, the aim  of this paper is to do an analogous  work for the fibres of the ramified Prym map in low genus.

	Although some specific cases were considered previously in \cite{nr} and \cite{bcv}, a systematic study of the properties of the ramified Prym map in full generality starts with the work of Marcucci and Pirola \cite{mp}. Combining their results with the main theorems in \cite{mn} and \cite{naranjo-ortega}, the generic Torelli theorem is proved for all the cases where the dimension of the source $\mathcal R_{g, r}$ is smaller than the dimension of the target $\mathcal A_{g-1+\frac r2}^{\delta}$. In fact, recently, a global Torelli theorem has been announced for all $g$ and $r\ge 6$ (\cite{ikeda} for $g=1$ and \cite{naranjo-ortega2} for all $g$).
	
	In this paper we address to the opposite side of the study of the ramified Prym map: the structure of the generic fibre when
	\begin{equation}\label{disuguaglianza}
		\dim \mathcal R_{g, r}=3g-3+r > \dim \mathcal A_{g-1+\frac r2}^{\delta}=\frac{1}{2}(g-1+\frac{r}{2})(g+\frac{r}{2}).
	\end{equation}
	Notice that this inequality still holds in case $g=1$: using translations we can always assume that one of the branch points is in the origin, hence $\dim \mathcal R_{1, r}=1+(r-1)=r$.
	Condition \eqref{disuguaglianza} is only possible in six cases that will be considered along the paper: $r=2$ and $1\le g\le 4$ and $r=4$ and $1\le g \le 2$. The case $g=1, r=4$ was considered by Barth in his study of abelian surfaces with polarization of type $(1,2)$ (see \cite{Barth}).
	
	 In \cite{fgs} and \cite{gm} infinitely many examples of totally geodesic and of Shimura subvarieties of ${\mathcal A}_g$  generically contained in the Torelli locus have been constructed as fibres of ramified Prym maps. 
	In  \cite{moonen-special}, \cite{moonen-oort}, \cite{fgp}, \cite{fpp} examples of Shimura subvarieties of ${\mathcal A}_g$ generically contained in the Torelli locus have been constructed  as families of Jacobians of  Galois covers of ${\mathbb P}^1$ or of elliptic curves. Some of them are contained in fibres of ramified Prym maps.
	
	In particular, the images in ${\mathcal M}_2$ and in ${\mathcal M}_3$ of ${\mathcal R}_{1,2}$, respectively ${\mathcal R}_{1,4}$, are  the bielliptic loci and, in \cite{fgs}, it is proven  that the irreducible components of the fibres of the Prym maps ${\mathcal P}_{1,2}$, ${\mathcal P}_{1,4}$ yield totally geodesic curves in ${\mathcal A}_2$ and  ${\mathcal A}_3$ and countably many of them are Shimura curves. 
	Moreover in \cite{fgs} it is shown that family $(7) = (23) = (34)$ of \cite{fgp} is a fibre of the Prym map ${\mathcal P}_{1,4}$, which is a Shimura curve.

	In this paper (section 8) we give an explicit example of a totally geodesic curve which is an irreducible component of a fibre of the Prym map ${\mathcal P}_{1,2}$.

	It is worthy to mention that degree $2$ coverings ramified in $2$ points can be seen as the normalization of coverings of nodal curves. Beauville extended in \cite{beau} the classical Prym map to some coverings of stable curves (called ``admissible'' coverings) in such a way that the extended Prym map
	\[
	\overline {\mathcal P}_{g}:\overline {\mathcal R}_{g} \longrightarrow \mathcal A_{g-1}                                                                                                                                                                                                                                                                                                                                                            \]
	becomes proper (to simplify the notation, $\mathcal P_{g,0}$ and $\mathcal R_{g,0}$ are denoted by $ \mathcal P_{g} $ and $ \mathcal R_{g} $). Then  the moduli space $\mathcal{R}_{g-1,2}$ can be identified with an open set of a boundary divisor of $\overline{\mathcal{R}}_{g}$. With this strategy the mentioned works of Verra, Recillas and Donagi could help to understand the cases $r=2$ and $2\le g \le 4$. Unfortunately this way becomes cumbersome since the  intersection of the generic fibre with the boundary is usually difficult to be described. We have found (except for the case $g=4$) direct procedures to study the fibre mainly based on the bigonal construction (see \cite{donagi}) and the extended trigonal construction (see \cite{lange-ortega}). We overview both in the next section.  
	
	The results obtained in this paper can be summarized in the Theorem below. To state our theorem in the case $r=2$, $g=4$ we need to recall that Donagi found a birational map
	\[
	\kappa: \mathcal A_{4} \longrightarrow \mathcal {RC}^+,
	\]
	where $\mathcal {RC}^+$ is the moduli space of pairs $(V, \delta )$, $V$ being a smooth cubic threefold and $\delta$ an ``even'' $2$-torsion point in the intermediate Jacobian $JV$ (see \cite[section 5]{donagi}).
	
	\begin{teo}
		We have the following description of the fibres of the ramified Prym map in the cases $(g,r) \in \{(1,2),(1,4),(2,2),(2,4),(3,2),(4,2)\}$:
		\begin{enumerate}  
			\item [a)] For a generic elliptic curve $E$ the fibre $\mathcal P_{1,2}^{-1}(E)$ is  isomorphic to $L_1 \sqcup \ldots \sqcup L_4$, where each $L_i$ is the complement of three points in a projective line.
			\item [b)] (Barth) Let $(A,L)$ be a generic abelian surface with a polarization of type $(1,2)$. Then there is a natural polarization $L ^* $ of type $(1,2)$ in the dual abelian variety $A^*$ and the fibre $\mathcal P_{1,4}^{-1}(A)$ is canonically isomorphic to the linear system $\vert L^*\vert$.
			\item [c)] The generic fibre of $\mathcal P_{2,2}$ is isomorphic to the complement of $15$ lines in a projective plane.
			\item [d)] The generic fibre of $\mathcal P_{2,4}$ is isomorphic to the complement of $15$ points in an elliptic curve.
			\item [e)] Let $X$ be a generic quartic plane curve, consider the variety $\mathsf G^1_4(X)$ of the $g^1_4$ linear series on $X$, and denote by $i$ the involution $L\mapsto \omega_X^{\otimes2}\otimes L\meno$. Then  $\mathcal P_{3,2}^{-1}(JX)$ is isomorphic to the quotient by $i$ of an explicit $i$-invariant open subset of $\mathsf G^1_4(X)$.
			\item [f)] Let $(V, \delta)$ be a generic element in $\mathcal {RC}^+$ and let $\Gamma \subset JV$ be the curve of lines  $l$ in $V$ such that there is a $2$-plane $\Pi$ containing $l$ with $\Pi \cdot V=l+2r$. Then $\mathcal P_{4,2}^{-1}(V,\delta )$ is isomorphic to  a precise open set $\Gamma_0$ of  $\Gamma $ (see \ref{open_set}).
		\end{enumerate}
	\end{teo}

	For more details  see Theorems \eqref{teo1}, \eqref{teo2}, \eqref{teo3}, \eqref{teo4}, \eqref{teo5}, \eqref{teo6}. 
	
	The paper is organized as follows: in the next section we give some preliminaries about the bigonal and trigonal construction and also on the differential of the ramified Prym map. In particular, we prove that the maps we are considering are dominant. Next we devote one section for each of the six cases. We include for completeness a description of Barth's results for $g=1$ and $r=4$. The most involved cases are $g=3$, $r=2$ and $g=4$, $r=2$. By means of the trigonal construction, the case $g=3$, $r=2$ is related with the determination of the tetragonal series on a generic quartic plane curve which do not contain divisors of type $2p+2q$. This is studied in detail in section 6. In the case $g=4$, $r=2$ we need to take care of the behaviour at the boundary of the rather sophisticated Donagi's description of the fibre of $\overline {\mathcal P}_{5}$ (see \cite[section 5]{donagi}). In particular we have to take care of the quadrics containing a nodal canonical curve of genus $5$ which we consider interesting on its own. This is the content of section 7. In section 8 we describe some examples of irreducible components of fibres of 	ramified Prym maps that yield totally geodesic or Shimura subvarieties of ${\mathcal A}_g$. \\

	\textbf{Acknowledgements:} The authors would like to thank A. Verra for his suggestions and references concerning section 6. We also thank Andr\'es Rojas for detecting a mistake in section $7$ of a previous version. The third author thanks IMUB (Institut de Matemàtica Universitat de Barcelona) for the hospitality it offered when the first draft of this project started. 
	
	\section{Preliminaries}
	
	\subsection{The differential of the ramified Prym map}\label{DIffeCov}
	
	By the theory of double coverings, the moduli space $\mathcal R_{g,r}$ can be alternatively described as 
	the following moduli space of triples $(C, \eta, B)$:
	\[
	\mathcal R_{g,r}=\{ (C, \eta,  B)  \mid  [C] \in \mathcal M_g,   \eta \in Pic^{\frac{r}{2}}(C),  B \text{ reduced divisor in } |\eta^{\otimes 2}|  \}/\cong.
	\]
	The codifferential of $\mathcal P_{g,r}$ at a point  $[(C, \eta, B)]\in \mathcal R_{g,r}$  is given by the multiplication map (\cite{mp})
	$$
	d\mathcal P_{g,r}^* (C, \eta, B): Sym^2 H^0(C, \omega_C \otimes \eta) \lra H^0(C, \omega_C^2 \otimes \mathcal O(B)).
	$$
	Let us now recall the definition of admissible covers given by Beauville in \cite[Lemma 3.1]{beau}.
		\begin{defin}
			Let $\tilde C$ be a connected curve with only ordinary double points and arithmetic genus $2g-1$, and let $\sigma$ be an involution on $\tilde C$. Then $\tilde C\ra \tilde C/\sigma$ is an admissible covering of type $(\ast)$ if the fixed points of $\sigma$ are exactly the singular points and at a singular point the two branches are not exchanged under $\sigma$.
		\end{defin}
		Under these conditions, Beauville shows that the Prym variety attached to the covering $\tilde C\ra \tilde C/\sigma$ can be defined in a similar way to the standard Prym construction and it is a principally polarized abelian variety.
	
	Let $\cov$ be an element of  $\mathcal R_{g,2}$. By glueing in $C$ the two branch points and in $D$ the two ramification points we get an instance of admissible covering of type $(\ast)$ of $\bar{\mathcal{R}}_{g+1}$.

	\begin{prop}
	\label{dominant}
		Assume that 
		\[
		(g,r) \in \{(1,2),(1,4),(2,2),(2,4),(3,2),(4,2)\},
		\]
		then the ramified Prym map  $\mathcal P_{g,r}$ is dominant. 
	\end{prop}
	\begin{proof}
	It is enough to show that for a generic $(C,\eta, B)$ there are no quadrics containing the image of $\varphi_{\omega_C \otimes \eta }:C \ra \mathbb P H^0(C, \omega_C \otimes \eta )^*$. Notice that there is nothing to prove in the cases $(g,r)\in\{(1,2), (2,2), (1,4)\}$. For $(g,r)=(3,2)$ and $(g,r)=(2,4)$ the curve 
	$\varphi_{\omega_C \otimes \eta }(C) $ is a plane curve (with nodes) of degree $5$ and $4$ respectively. For the case $(g,r)=(4,2)$ we identify elements $(C,\eta, p+q)\in\mathcal{R}_{4,2}$ with coverings $(C^*=C/p\sim q ,\eta^*)$ of type $(\ast)$ in $\overline{\mathcal{R}}_5$.

	Izadi proved in  \cite[Theorem (3.3), Remark (3.10)]{izadi}  that the fiber at an abelian fourfold $A$, out of the closure of the Jacobian locus and of some specific high codimension locus, intersects the boundary in dimension $1$. Therefore, by dimensional reasons, either $\Delta^n$ dominates $\mathcal A_4$ or maps in the Jacobian locus. The fiber at a generic Jacobian is given in \cite[Theorem (5.14), (4)]{donagi}. It is formed by two surfaces, one in the  locus of coverings of trigonal curves and the other in the divisor of Wirtinger coverings. Hence the $11$-dimensional divisor $\Delta^n$ cannot be mapped to the $9$-dimensional Jacobian locus.\\
	\end{proof}

	\begin{cor}
		The assumptions of the previous proposition imply that the dimension of the generic fibre $F_{g,r}$ of $\mathcal P_{g,r}$ is:
		\[
		\begin{aligned}
		&\dim F_{1,2}=1, \, \dim F_{2,2}=2, \,\dim F_{3,2}=2, \\ &\dim F_{4,2}=1, \,\dim F_{1,4}=1, \,\dim F_{2,4}=1.   
		\end{aligned}
		\]
	\end{cor}
\subsection{Dual Abelian Variety}\label{dualpol}
Here we recall the main result of Birkenhake-Lange concerning dual abelian varieties and dual polarizations. 
\begin{teo}[\cite{birkenhake-langepol}, Theorem 3.1]
	There is a canonical isomorphism of
	coarse moduli spaces
	\begin{align}\label{lbiso}
	\mathcal{A}_g^{(d_1,...d_g)}&\ra\mathcal{A}_g^{(\frac{d_1d_g}{d_g},\frac{d_1d_g}{d_{g-1}},...,\frac{d_1d_g}{d_1})}\\
	(A,L)&\mapsto (A^*,L^*)\notag
	\end{align} 
	sending a polarized abelian variety to its polarized dual abelian
	variety.
\end{teo}
Here $L^*$  is the polarization on the dual abelian variety $A^*$ which satisfies \[\lambda_{L^*}\circ\lambda_{L}=(d_g)_{A} \quad \text{and}\quad \lambda_L\circ\lambda_{L^*}=(d_g)_{A^*},\]
where $\lambda_L: A\ra A^* $, $\lambda_{L^*}: A^*\ra (A^*)^*=A$ are the polarization maps and $ (d_g)_{A}:A\ra A $, $ (d_g)_{A^*}:A^*\ra A^*$ are the multiplications by $d_g$. 

Notice that the dual polarization $L^*$ satisfies $(L^*)^*=L$.
\subsection{The polygonal construction}\label{polygonal}
This section is devoted to the description of the so-called polygonal construction. It provides a very useful tool which starts from a ``tower'' of coverings
\[A\ra B\ra C\]and produces new ones: $ A'\ra B'\ra C', A''\ra B''\ra C'',...$ 
determining relations among the Prym varieties. All details of this construction are borrowed from \cite{donagi}.\\

Let us consider a curve $C$ of genus $g$ with a map  $f:C\to\mathbb{P}^1$  of degree $n$ and a $2$-sheeted ramified covering $\cov$. Then we can always associate a $2^n$- covering
\begin{equation*}
D'\to \mathbb{P}^1
\end{equation*}
defined in the following way: the fibre over a point $p\in \mathbb{P}^1$ is given by the $2^n$ sections $s$ of $\pi$ over $p$. This means that:
\begin{equation}\label{sezioni}
s: f^{-1}(p)\to\pi^{-1}f^{-1}(p) \quad \text{and}\; \pi\comp s=id .
\end{equation} 
$D'$ can be better described inside $D^{(n)}$, where $D^{(n)}$, as usual, represents the $n$-symmetric product of the curve $D$ and it parametrizes effective divisors of degree $n$. Indeed it can be described by the following fibre product diagram:
\begin{equation}\label{poligconstr}
\begin{tikzcd}
D'\arrow{d}{2^n:1}\arrow[hook]{r}&D^{(n)}\arrow{d}{\pi^{(n)}}\\
\PP\arrow[hook]{r}&C^{(n)}
\end{tikzcd}
\end{equation}
where $\PP$ is embedded in $C^{(n)}$  by sending a point $ p $ to its fibre $f\meno(p)$.

$D'$ carries a natural involution $i':D'\to D'$ defined as follows: \begin{equation}\label{involutione}
q_1+...+q_n\mapsto i(q_1)+...+i(q_n)
\end{equation} where $i$ is the involution of $D$ which induces the covering $\pi$. Moreover we can define an equivalence relation on $D'$ identifying two sections $s_i,s_j$ if they correspond under an even number of changes $q_i\mapsto i(q_i)$. This gives another tower 
\begin{equation*}
D'\ra O\ra \PP,
\end{equation*}
where $O$ is the quotient obtained considering this equivalence. It is known as the orientation cover of $f\circ \pi$. 

We conclude recalling, without proof, a result shown in \cite{donagi}.
\begin{prop}\label{Reducible}
	If $D$ is \textit{orientable}, that means that the orientation cover $O\ra\PP$ is trivial, then $D'$ is reducible: $D'=D^0\cup D^1$.
	\begin{itemize}
		\item[-] If $n$ is even then $i'$ acts on each $D^j$ and the quotient has a degree $2^{n-2}$ map to $\PP$;
		\item[-] If $n$ is odd then $i'$ exchanges the two branches $D^j$. Each $D^j$ has a map of degree $2^{n-1}$ to $\PP$.
	\end{itemize}
\end{prop}
\subsection{The bigonal construction}\label{bigonal}
Let us see an application of the polygonal construction described above in case of  $n=2$. Starting from a tower 
\[D\xrightarrow{\pi} C\xrightarrow{f}\PP,\] where $\pi$ and $f$ both have degree 2, we get \[D'\xrightarrow{\pi'} C'\xrightarrow{f'}\PP,\] by means of a fibre product diagram as in diagram \eqref{poligconstr}. Taking $k\in\PP$ the possible situations are the following (see \cite{donagi}, pp. 68-69):
\begin{itemize}
	\item[1)] If $\pi,f$ are \'{e}tale the same are $\pi',f'$;
	\item[2)] If $f$ is \'{e}tale while $\pi$ is branched at one point of $f^{-1}(k)$, then $h$ inherits two critical points of order 2 in the fibre which are exchanged by $i$. This means that $\pi'$ is \'{e}tale, while $f'$ is branched;
	\item[3)] Viceversa if $\pi$ is \'{e}tale while $f$ is branched in $k$, then $h$ has a critical point of order 2 in the fibre and 2 more points which are exchanged by $i'$. This means that $\pi'$ has a critical point of order 2 while $f'$ is \'{e}tale;
	\item[4)] If $\pi,f$ are both branched the same are $\pi',f'$ (in particular $h$ has a single critical point of order 4);
	\item[5)] If $f$ is \'{e}tale while $\pi$ is branched at both points then $C'$ will have a node over $k$. 
	
\end{itemize}
We call an element $D\ra C\ra\PP$ \textit{general} if it avoids situations of type 5) (where the bigonal construction induces singular coverings).
\begin{prop}
	Assuming $f\circ\pi$ general,
	then $g(D')=r+g-2$ and $g(C')=\frac{r}{2}-1$.
	\begin{proof}
		By the assumption of generality, a straightforward application of Riemann-Hurwitz formula for $h$ gives:
		\begin{equation*}
		2g(D')-2=4(-2)+2(r-r')+6-r'+3r'.
		\end{equation*}
		Similarly for $\pi'$, we get: \begin{equation*}
		2g(D')-2=2(2g(C')-2)+6-r'+r'.
		\end{equation*}
	\end{proof}
\end{prop}

\begin{lemma}[\cite{donagi}, Lemma 2.7]
	The bigonal construction is symmetric: if it takes $D\xrightarrow{\pi} C\xrightarrow{f}\PP$ to $	D'\xrightarrow{\pi'} C'\xrightarrow{f'}\PP$ then it takes $	D'\xrightarrow{\pi'} C'\xrightarrow{f'}\PP$ to $	D\xrightarrow{\pi} C\xrightarrow{f}\PP$.
\end{lemma}
Moreover the following holds:
\begin{teo}[Pantazis,\cite{Pantazis}]\label{Pantazis}
	The Prym varieties $P(D,C)$ and $P(D',C')$ associated to the two bigonally-related covering maps $D\ra C$ and $D'\ra C'$ are dual each other as polarized abelian varieties. 
\end{teo}
\subsection{The trigonal construction}
Thanks to the work of Recillas (\cite{ReciTrig}), we have a theorem concerning the polygonal construction in the case $n=3$. It is known as \textit{trigonal construction} and it deals with \'{e}tale double covers of smooth trigonal curves. 

Denote $\mathcal{R}_{g+1}^{tr}$ the moduli space of 2:1 \'{e}tale coverings of trigonal curves $\tilde C$ of genus $g+1$. Each point in $\mathcal{R}_{g+1}^{tr}$ corresponds to a triple $(\tilde C, \eta, M)$, where $\eta\in\text{Pic}^0(\tilde C)$ such that $\eta\neq\OO_C$ and $\eta^2=\OO_{\tilde C}$ gives the double covering and $M$ is the $g^1_3$ which gives the map to $\PP.$ This means that we consider towers \begin{equation}
\tilde D\xrightarrow{\tilde \pi}\tilde C\xrightarrow{3:1} \PP.
\end{equation}
Now call $\mathcal{M}^{tet}_{g,0}$ the locus in $\mathcal{M}_g$ given by tetragonal curves $X$ with the property that above each point of $\PP$ the associated linear series $g^1_4$ has at least one \'{e}tale point. In \cite{ReciTrig} Recillas showed (see \cite{lange-birkenhake} for details of the construction) that: \begin{teo}The trigonal construction gives the following isomorphism:\begin{align*}
	T_0: \mathcal{R}_{g+1}^{tr}&\ra\mathcal{M}^{tet}_{g,0}\\
	(\tilde C, \eta, M)&\mapsto(X,F),
	\end{align*}
	where $F$ is a $g^1_4$.\\
	Moreover, calling $P(\tilde \pi)$ the Prym variety associated to $\tilde \pi$, we have: \begin{equation*}
	P(\tilde \pi)\cong JX \label{Isom}
	\end{equation*}
	as isomorphism of principally polarized abelian variety (from now on denoted by ppav). 
\end{teo}
Notice that the polygonal construction in case $n=3$, applied to an unbranched covering $\pi$, gives a reducible curve $X=X^0\cup X^1$ (see \cite[Corollary 2.2]{donagi}). The two components are isomorphic tetragonal curves of genus $g$. Take one of them to define the image $(X,F)$ of $(	\tilde C, \eta, M)$ through $T_0$. \\

A similar statement is valid also in the case of double covers of trigonal curves with two ramification points. This has been proved by Lange and Ortega in \cite{lange-ortega}. Let $\mathcal{R}b_g^{tr}$ be the moduli space of pairs $ (\cov, M )$ where $\pi$ is a ramified double cover of a smooth trigonal curve $C$ of genus $g$ and $M$ is a $g^1_3$ on $C$. Suppose that the branch locus of $\pi$ is disjoint from the ramification locus of the degree 3 map $f:C\ra\PP$. 
 As in \cite{lange-ortega}, we will call an element $D\xrightarrow{\pi}C\xrightarrow{f}\PP$ \textit{special} if the branch locus of $\pi$ (given by two points $p_1,p_2$) is contained in a fibre of $f$, otherwise we will call it \textit{general}. Let $\mathcal{R}b_{g,sp}^{tr}$ be the moduli space of special elements. 

Moreover we denote $\mathcal{M}_{g,*}^{tet}$ the moduli space of pairs $(X,k)$ of smooth tetragonal curves with a 4:1 map $k:X\ra\PP$ with at least one \'{e}tale point on each fibre with the exception of exactly one fibre which consists of two simple ramification points. 
\begin{teo}[\cite{lange-ortega}, Theorem 4.3]\label{ramifiedtrigonal}
	The map \begin{equation}
	\mathcal{R}b_{g,sp}^{tr}\ra\mathcal{M}_{g,\ast}^{tet}\label{IsoLangeOrtega}
	\end{equation} is an isomorphism. Moreover if $D\xrightarrow{\pi}C\xrightarrow{f}\PP$ is an element of $\mathcal{R}b_{g,sp}^{tr}$ and $X$ is the corresponding tetragonal curve, then we have an isomorphism of ppav:\begin{equation*}
	P(\pi)\cong JX.
	\end{equation*}
\end{teo} 

\section{Case $g=1$, $r=2$}

 Let us consider $\pi: D\ra C$ a double ramified covering in $ \mathcal{R}_{1,2}$ and take $$\mathcal{P}_{1,2}: \mathcal{R}_{1,2}\ra\mathcal{A}_1,$$ the corresponding Prym map.
 
  We denote by $b_1+b_2$ the branch divisor (on $C$) and $r_1+r_2$ the ramification divisor (on $D$). The covering $\pi$ is determined by the data $(C,\eta, b_1+b_2)$ where $\eta \in Pic^1(C)=C$ satisfies $ \eta^{\otimes 2} =\mathcal O_C(b_1+b_2)$. The linear series $\vert b_1+b_2 \vert $ gives a map $f:C\xrightarrow{2:1}\mathbb P^1$ ramified in four points $p_1,p_2,p_3,p_4 \in C$. By construction $\eta $ is one of the sheaves $\mathcal O_C(p_i)$.
  
   Calling $\sigma $ the involution on $C$ attached to $f$, then $$\sigma(b_1)=b_2 \quad \text{and}\quad \sigma(p_i)=p_i,\;\; i=1,...,4.$$ Hence $\sigma$ leaves invariant $b_1+b_2$ and $\eta$. The following is well-known:
 \begin{lemma}
  Let $\sigma $ be an involution on a curve $C$ leaving invariant a reduced divisor $B$ and a sheaf $\eta \in Pic(C)$ such that $\eta ^{\otimes 2}\cong \mathcal O_C(B)$. Let $\pi: D\ra C$ be the double covering attached to $(C,\eta, B)$, then there exists an involution $\tilde \sigma $ on $D$ lifting $\sigma$, that is $\sigma \circ \pi= \pi \circ \tilde \sigma $.
  \begin{proof} The curve $D$ is defined as $D:= Spec (\mathbf{\mathcal A})$, where $\mathbf{\mathcal A}$ is the $\mathcal O_C$-algebra $\mathcal O_C \oplus \eta ^{-1}$ and the multiplication is defined in the obvious way using that $\eta ^{-2}\cong \mathcal O_C(-B)\subset \mathcal O_C$ (see \cite[section 1]{mumford}). Then $\sigma $ induces an involution on $\mathbf {\mathcal A}$ and therefore an involution $\tilde \sigma $ on $D$ which lifts $\sigma $ by construction. \\\end{proof} 
   \end{lemma}
 We have the following Cartesian diagram:
 \begin{equation}
 \begin{tikzcd}
 & D\arrow{dl}\arrow{dr}\arrow[d,"\pi"]\\
 \PP\arrow{dr} &C \arrow[d,"f"]&E\arrow[dl,"f'"]\\
 &\PP
 \end{tikzcd}
 \end{equation}
Here $E$ is the quotient of $D$ by $\tilde \sigma$ while $\PP$ is the quotient of $D$ by $\tau\tilde{\sigma}$, where $\tau$ is the involution attached to $\pi$. Indeed, without loss of generality, assume that $\pi$ corresponds to the point $p_1$ (i.e. $\eta \cong \mathcal O_C(p_1)$). Then  the preimages of $p_1 $ by $\pi$ are  fixed points of $\tilde \sigma $ and they are the only ones. On the other hand, $\tau\tilde{\sigma}$ fixes the preimages by $\pi$ of $p_2,p_3,p_4$. Therefore $E$ is an elliptic curve, while $D/\langle\tau\tilde{\sigma}\rangle$ is $\PP$ (note that if this is not true the contrary holds). Using \cite[section 7]{mumford}, we get $P(D,C)\cong E$. \\

Calling $a_i=f(p_i)$ and $b=f(b_1)=f(b_2)$, we obtain that the branch locus of $f': E\ra \PP$ is given by $b, a_2, a_3, a_4$. We have the following: 
\begin{teo}\label{teo1}
 Fix a generic elliptic curve $E\in\mathcal{A}_1$. The preimage of $E$ by the ramified Prym map $\mathcal{P}_{1,2}$ is isomorphic to $L_1 \sqcup \ldots \sqcup L_4$, where each $L_i$ is the complement of three points in a projective line.
 
 \begin{proof}
 	Start with $E$ represented as a double covering of $\mathbb P^1$ branched in four points $c_1, c_2, c_3, c_4$ and put $$L_i=\mathbb P^1\backslash \{ c_{1},...,\hat{c}_{i},...,c_{4}\}.$$ Then for any $q\in L_1$ we get a unique element in $\mathcal{P}_{1,2}\meno(E)$ in the following way: $C$ is the covering of $\mathbb P^1$ branched in $q, c_2, c_3, c_4$. Denote with $b_1,b_2$ the preimages of $c_1$ via this covering. Then $D\ra C$ is determined by $b_1+b_2$ and by $\eta=\OO_C(p)$, where $p$ is the ramification point in $C$ attached to $q$. \\
 	Doing the same for the other $L_i's$, we conclude.\\
 \end{proof}
\end{teo}
\section{Case $g=1$, $r=4$}
The case \begin{equation}\label{Prym1,4}
\mathcal{P}_{1,4}: \mathcal{R}_{1,4}\to \mathcal{A}_2^{(1,2)}
\end{equation} is completely studied in \cite{Barth}. Here we include the main result without proof by the sake of completeness.\\

Actually, instead of \eqref{Prym1,4}, it is easier to study the composition:
\begin{equation}
	\mathcal{R}_{1,4}\to \mathcal{A}_2^{(1,2)}\xrightarrow{\cong}\mathcal{A}_2^{(1,2)}
\end{equation}
where the isomorphism sends the Prym variety to its dual (see subsection \ref{dualpol}). 

Fix a general polarized abelian surface $(A,L)$ of type (1,2). We have the following:
\begin{prop}[\cite{Barth}, pp. 46-48]
	The pencil $\vert L\vert $ has no fixed component and its base locus consists of four points $e_1,...,e_4$.  The general member $D\in \vert L\vert $ is an irreducible smooth curve of genus 3. Moreover, $L$ is symmetric and the same occurs for all $D\in \vert L\vert$.
\end{prop}
Furthermore if $D \in |L|$ is smooth, the multiplication by $-1$ has exactly 4 fixed points on it. This means that the quotient $D/\langle -1 \rangle$ is an elliptic curve. 
\begin{teo}[Duality Theorem 1.12,\cite{Barth}]\label{teo2}
	The fibre of the Prym map is parametrized by the linear system $|L^*|$, where $L^*$ is the dual polarization of $L$ defined as in \ref{dualpol}.
\end{teo}

\section{Cases $g=2$ and $r=2$}
This section is devoted to the analysis of the fibres of
\[  
\mathcal{P}_{2,2}: \mathcal{R}_{2,2}\to \mathcal{A}_2.
\]
Taking an element $(C,\eta, B)\in \mathcal{R}_{2,2} $, we can apply the bigonal construction using the hyperelliptic involution of $C$. Thus we pass from towers:
\begin{equation*}
\label{startbigonal}
D\ra C\ra\PP
\end{equation*}
to towers
\begin{equation*}\label{resbigonal}
D'\xrightarrow{\pi'} C'\xrightarrow{f'}\PP,
\end{equation*}
where $\pi'$ is a degree 2 map branched on $6$ points and $C'$ is a curve of genus $0$ maybe with one node. In order to understand the nodal case the map $f'$ can be seen as choosing two different points in the projective line, the limit case appears when the two points come together.

Denote by $\mathcal M_{0,6}$ the moduli space of 6 unordered different points in the projective line and by $\mathcal M_{0,6,2}$ the moduli space of two collections of points in the line: $6$ unordered different points and $2$ unordered different points more. A partial compactification of this last space $\overline {\mathcal M}_{0,6,2}$ consists in allowing the set of two points to be a repeated one. 

As described in subsection \ref{bigonal}, the bigonal map yields an injective map:
\[
b: \mathcal R_{2,2} \ra \overline {\mathcal{M}}_{0,6,2}.
\]
This is an isomorphism of $\mathcal R_{2,2}$ with $b(\mathcal R_{2,2})$, where the inverse map is the bigonal construction again (according to \cite[Section 2.3]{donagi}, more precisely possibility (vi) in page 69, the bigonal map extends to nodal admissible coverings). 
We denote the image $b(\mathcal R_{2,2})$ by $\overline {\mathcal M}_{0,6,2}^{0}$ .

 In order to give a precise description of this moduli space we identify the symmetric product $Sym^2 \mathbb P^1$ with a projective plane in the standard way: we see $\PP\hookrightarrow\mathbb{P}^2$ as a conic via the Veronese embedding of degree $2$. The pairs of points (possibly equal) correspond to lines and therefore $Sym^2 \mathbb P^1$ can be identified with  
$\mathbb P^{2 \vee}$. The way that each  pair of different points $x_1,x_2$ determines a $2:1$ map on the conic is easy: two points  $z_1,z_2$ correspond by the involution $\sigma_{x_1+x_2}$ with fixed points $x_1,x_2$, if they form  a harmonic ratio: $\vert x_1,x_2;z_1,z_2 \vert =-1$. Geometrically, viewing the points in the plane, this means that the pole of the line $x_1x_2$ is aligned with $z_1$ and $z_2$.\\

We have $6$ marked points $p_1,...,p_6$ in the line and we have to avoid property (5) of subsection \ref{bigonal}: we have to eliminate the pairs $x_1 + x_2 \in Sym^2\mathbb P^1$ such that $\sigma _{x_1+x_2}(p_i)=p_j$ for some $i\neq j$.
Hence:
\[
\overline {\mathcal M}_{0,6,2}^{0}=\{[(p_1+\ldots +p_6,x_1+x_2)]\in \overline {\mathcal M}_{0,6,2} \mid \vert x_1,x_2;p_i,p_j\vert \neq -1, \forall i\neq j \}.
\]

Then we have a commutative diagram:
\begin{equation}
\begin{tikzcd}\label{Fibra2,2}
\mathcal{R}_{2,2}\arrow{r}\arrow{d}{b}	& \mathcal{A}_2\\
\overline {\mathcal{M}}_{0,6,2}^0 \arrow{r}{\phi} &\mathcal{M}_{0,6}\arrow{u}
\end{tikzcd}
\end{equation}
where $\phi$ is the forgetful map and $\mathcal{M}_{0,6} \ra \mathcal{A}_2$ is just the Torelli morphism. 
Therefore, studying the fibre of $\phi$, we conclude with the following:
\begin{teo}\label{teo3}
	The fibre of the Prym map $\mathcal{P}_{2,2}$ over a general principally polarized abelian surface $S$ is isomorphic to a projective plane minus $15$ lines.
	\begin{proof}
		Let $S$ be a general principally polarized abelian surface, assume that $S$ is the Jacobian of a genus $2$ curve $H$ and represent $H$ as an element in $\mathcal{M}_{0,6}$, where the six marked points $p_1,...,p_6$ are all different and correspond to the branch locus of the hyperelliptic involution. Diagram \eqref{Fibra2,2} says that we must look at the fibre of $\phi$ over $H$. The harmonic condition $\vert x_1,x_2;p_i,p_j\vert = -1$ says that $x_1$, $x_2 $ and the pole $p_{i j}$ of $p_ip_j$ are in a line. Therefore, looking at the dual, we have to rule out the points of the $15$ lines $(p_{i j})^*\subset \mathbb P^{2 \vee}$. 
		
		Notice that the limit case $x_1=x_2$ means that $p_{i j}$ belongs to the tangent line at the point and this is not excluded in the fibre. \\
	\end{proof}
\end{teo}

\section{Case $g=2$ and $r=4$} 
We study now the map
$$\mathcal{P}_{2,4}: \mathcal{R}_{2,4}\to \mathcal{A}_3. $$
In this case the bigonal construction produces an injective map: 
\[
b: \mathcal{R}_{2,4} \ra  \tilde{\mathcal{R}}_{1,6},
\] 
where $\tilde{\mathcal{R}}_{1,6} $ is the moduli space of isomorphism classes of pairs $(\pi',f')$. Here $\pi':D' \ra C'$ is a degree 2 map branched on 6 smooth points  $p_1,...,p_6$. The curve $D'$ has at most one node, fixed by the involution on $D'$ attached to $\pi'$, which determines an admissible singularity for $\pi'$ of type $(\ast)$ by \cite[Section 2.3, possibility (vi)]{donagi}.  The map $f'$ is a $g^1_2$ on $C'$. 

We denote by $\tilde{\mathcal{R}}_{1,6}^0 $ the image of $b$. This is an open set that, as in the previous section, can be described explicitly 
\[
\tilde{\mathcal{R}}_{1,6}^0=\{ (\pi',f')\in \tilde {\mathcal R}_{1,6} \mid f'(p_i)\neq f'(p_j) \;\forall i\neq j \}.
\]

The symmetry of the bigonal construction makes $b$ an isomorphism onto its image.\begin{remark}
		In order to state the next diagram, we need first to extend $\mathcal P_{1,6}$ to the partial compactification $\mathcal R_{1,6}'$ of the double coverings of curves of (arithmetic) genus $1$ which satisfy our assumptions on $\pi'$. This is possible, although we are working with ramified coverings, doing a local analysis around the singular points of  $\pi'$ and imitating Beauville's construction of the extension of the Prym map to admissible coverings. Indeed both $JD'$ and $JC'$ turn out to be $\mathbb{C}^*$-extensions of abelian varieties and a diagram similar to that in \cite{beau} pag. 174 shows that the kernel of the Norm map induced by $\pi'$ (that is $P(D'\ra C')$) is an abelian variety. 
\end{remark}

Similarly to the previous section we have the following diagram:
\begin{equation}
\begin{tikzcd}\label{Fibra2,4}
\mathcal{R}_{2,4}\arrow{r}\arrow{d}{b}	& \mathcal{A}_3^{(1,2,2)}\cong\mathcal{A}_3^{(1,1,2)}\\
\tilde{\mathcal{R}}_{1,6}^0\arrow{r}{\phi} &\mathcal{R}_{1,6}'\arrow{u}
\end{tikzcd}
\end{equation}
where $\phi$ is the forgetful map. The isomorphism $\mathcal{A}_3^{(1,2,2)}\cong\mathcal{A}_3^{(1,1,2)}$ is given by \eqref{lbiso} and sends a polarized abelian threefold to its dual (endowed with the dual polarization). The remaining vertical arrow $\mathcal{P}_{1,6}': \mathcal{R}_{1,6}' \ra \mathcal{A}_3^{(1,1,2)} $ is the extension of the  Prym map $\mathcal P_{1,6}$. 
From a result of Ikeda (\cite{ikeda}), we know that  $\mathcal{P}_{1,6}$ is injective and in fact an embedding (see \cite{naranjo-ortega2}). Therefore the extension to $\mathcal R_{1,6}'$ is generically injective. As before, Theorem \eqref{Pantazis} guarantees the commutativity of \eqref{Fibra2,4}. We conclude with the following:
\begin{teo}\label{teo4}
	The fibre of the Prym map $\mathcal{P}_{2,4}$ over a general $A\in\mathcal{A}_3$ is isomorphic to an elliptic curve $E$ minus 15 points.
	\begin{proof}
		Let us consider a generic polarized abelian threefold in $\mathcal{A}_3^{(1,2,2)} $ and let $(E, \eta, p_1+\ldots + p_6)$ be its unique preimage  in $\mathcal R'_{1,6}$. Call $B=p_1+...+p_6$ the branch divisor . Diagram \eqref{Fibra2,4} says that the fibre over $A$ is isomorphic to the fibre of $\phi$ over  $(E, \eta, p_1+\ldots + p_6)$, hence it is isomorphic to: \[Pic^2(E)\setminus\bigcup_{\substack{p_i,p_j\in B,\\p_i\neq p_j}}\OO_E(p_i+p_j). \]
		The isomorphism $Pic^2(E)\cong E$ concludes the proof.\\
	\end{proof}
\end{teo}

\section{Case $g=3, r=2$}
Let us now look at the map \begin{equation}
		\mathcal{P}_{3,2}: \mathcal{R}_{3,2}\to \mathcal{A}_3.
\end{equation}
Note that here the associated Prym varieties are principally polarized. 

Each non-hyperelliptic curve of genus $3$ admits a $1$-dimensional space of $g^1_3$'s.  In particular, seeing $C$ as a quartic plane curve and considering the line $l=p_1+p_2$ passing through $p_1$ and $p_2$, we can always get two degree 3 maps: they are defined considering the two different projections from one of the two remaining points $x,y$ of the intersection $C\cdot l$.
In fact if we consider the canonical divisor $$K_C=p_1+p_2+x+y$$ we get $h^0(C,\omega_C(-x))=h^0(C,\omega_C(-y))=2$ and we can use the associated linear systems to define the 3:1 maps to $\PP$. Call them $f_x$ and $f_y$. Both have, by definition, the two branch points $p_1,p_2$ on the same fibre and they are the unique trigonal maps on $C$ with this property. \\

We will use the following diagram to describe the fibres of $\mathcal{P}_{3,2}$:
\begin{equation}\label{Diag}
	\begin{tikzcd}
	&\mathcal{M}_{3,\ast}^{tet}\arrow{r}{2:1}&\mathcal{M}q_{3,\ast}^{tet}\arrow{dr}\\
	\mathcal{R}b_{3,sp}^{tr}\arrow{d}{2:1}\arrow{ur}{\cong}& & &\mathcal{M}_{3}\arrow{d}{j}\\
	\mathcal{R}_{3,2} \arrow{rrr}{\mathcal{P}_{3,2}}& & &\mathcal{A}_3
	\end{tikzcd}
\end{equation}
Let $\mathcal{R}b_{3,sp}^{tr}$ be the moduli space of pairs $(\cov, M)$, where $\pi$ is a ramified double cover of a smooth trigonal curve $C$ of genus 3, $M$ is a $g^1_3$ on $C$ such that the branch locus of $\pi$ is contained in one of its fibres. By above considerations, the forgetful map $\mathcal{R}b_{3,sp}^{tr}\ra\mathcal{R}_{3,2} $ is a 2:1 map.
The ramified trigonal construction (as recalled in \eqref{ramifiedtrigonal}) gives the isomorphism between $\mathcal{R}b_{3,sp}^{tr}$ and $\mathcal{M}_{3,\ast}^{tet}$ . \\
We will study the fibre of $\mathcal{P}_{3,2}$ using the map $	\mathcal{M}_{3,\ast}^{tet}\ra \mathcal{M}_{3}$. Note that in the above diagram it factors as the composition of two maps. The first one is $		\mathcal{M}_{3,\ast}^{tet}\ra	\mathcal{M}q_{3,\ast}^{tet}$ defined as the quotient map associated with an involution that acts on $\mathcal{M}_{3,\ast}^{tet}$. We will describe this action later. The second map $	\mathcal{M}q_{3,\ast}^{tet}\ra \mathcal{M}_{3}$ is the forgetful map. Finally,
$j$ is just the Torelli morphism.\\
 
Let us take a general abelian threefold $A\in\mathcal{A}_3$. We can assume that $A$ is the Jacobian of a general curve $X$ of genus $3$. In order to study the fibres of 
$	\mathcal{M}_{3,\ast}^{tet}\ra \mathcal{M}_{3}$ we need to recall some facts on $g^1_4$'s on $X$.
\subsection{The blow-up}

Let $X\subset \mathbb P^2=\mathbb P (H^0(X,\omega_X)^{*}) $ be a non-hyperelliptic curve of genus $3$ canonically embedded. 
Let $\mathsf G^1_4(X)$ be the variety of all $g^1_4$ linear series on $X$, complete or not (see \cite{ACGH}, chapter IV). Then by Riemann-Roch  
\[
\psi:\mathsf G^1_4(X) \lra W^1_4(X) =Pic ^4(X)
\]
is a birational surjective map which is an isomorphism out of $W^2_4(X)=\{\omega_X\}$. In fact $$\text{Supp}(\mathsf G^1_4(X))=\{(L,V)\mid L\in Pic ^4(X), V\in Gr(2,H^0(X,L))\}.$$ Thus over $L\neq \omega_X$ the fibre is just the complete linear series $(L,H^0(X,L))$. Call $E$ the preimage of the canonical sheaf, 
then $\mathsf G^1_4(X)\smallsetminus E \cong Pic^4(X)\smallsetminus \{\omega_X\}$. The set $E$ parametrizes all the non-complete $g^1_4$ linear series on $X$ which correspond to 
\[
Gr(2, H^0(X
,\omega_X))\cong \{ \text{lines in }\mathbb P H^0(X,\omega_X)\} = \mathbb P (H^0(X, \omega_X)^*).
\]
In other words $\mathsf G^1_4(X)$ is the blow-up of $Pic^4(X)$ at $\omega_X$ and the points of the exceptional divisor correspond to points in the 
plane $\mathbb P^2$ where the curve $X$ is canonically embedded. 
The linear series is the projection from this point and if the point belongs to $X$ itself, then the linear series has a base point.\\

We can assume that $X$ has exactly $28$ bitangents, 
that is that there are not hyperflexes in $X$ (points $p$ such that the tangent line at $p$ intersects $X$ in $4p$). 
In fact, the curves with hyperflexes define a divisor in $\mathcal M_3$.  Each bitangent defines a divisor of the form $2p_i+2q_i$ 
in the canonical linear series of $X$. Denote by $\mathcal B \subset X^{(2)}$ the set $\{p_i+q_i \mid i=1,\ldots , 28\}$ 
and let $$S:=\text{Bl}_\mathcal{B}X^{(2)} $$ be the surface obtained by blowing-up $X^{(2)}$ at $\mathcal B$. 
By the universal property of the blow-up we have a diagram:

\begin{equation}\label{blowup}
	\begin{tikzcd}
	S   \ar[d] \ar{r}{\varphi }   & \mathsf {G}^1_4(X) \ar{d}{\psi}\\
	X^{(2)} \ar{r}{\varphi_0 }& Pic^4(X),
	\end{tikzcd}
\end{equation}
where $\varphi_0(x+y)=\mathcal O_X(2x+2y)$. 

Let us now consider the involution \begin{align*}
	i: Pic^4(X)&\ra  Pic^4(X)\\
	L&\mapsto \omega_X^{\otimes2}\otimes L\meno.
	\end{align*}
\begin{prop}\label{lift}
	The involution $i$ on $Pic^4(X)$ lifts to an involution on $\mathsf G^{1}_4 (X)$ and it acts as the identity on the exceptional divisor $E$. Moreover, by construction, $i$ leaves $\varphi (S)$ invariant.
	\begin{proof}
		To simplify the notation put $Pic= Pic^4(X)$. The involution $i$ has an isolated fixed point at $\omega_X$. In fact the fixed points are the line bundles $L$ such that $L^{\otimes2}=\omega_X^{\otimes2}$ and this happens if and only if $L =\omega_X\otimes \eta$, where $\eta$ is a two torsion point. 
		
		The exceptional divisor $E$ is equal to $ \mathbb P (T_{\omega_X}Pic)$ so the action of $i$ on $E$ is given by the projectivisation of the differential of $i$ at $\omega_X$, $di_{\omega_X}$. We claim that $di_{\omega_X}$ is $-Id$, hence it is the identity on $E = {\mathbb P}(T_{\omega_X}Pic)$. In fact by the linearisation theorem of Cartan, there exist local coordinates $z$ in a neighborhood $U$ of $\omega_X$ such that  in these coordinates, $i(z) = Az$, where $A$ is  a matrix such that $A^2 = I$. Thus the eigenvalues of $A$ are $\pm1$. But if there is an eigenvalue equal to 1, there would exist a space of fixed points which is positive dimensional. This leads to a contradiction. 
		
		Finally let us take $x+y\in X^{(2)}$, as in diagram \eqref{blowup}, and let us denote with $x',y'$ the two remaining points of the intersection of $X$ with the line $l=x+y$. In this way $K_X=x+y+x'+y'$ and thus the last statement follows from $i(\varphi_0(x+y))=\mathcal O_X(2x'+2y')$. 
		
	\end{proof}
\end{prop}
Proposition \ref{lift} guarantees that $i$ naturally induces an involution on $		\mathcal{M}_{3,\ast}^{tet}$ (the moduli space of pairs $(C,g^1_4)$ of curves $C$ of genus 3 with a 4:1 map to $\PP$ with at least an \'etale point on each fibre and a special fibre of type $2p+2q$). We still denote this involution by $i$.

\subsection{Geometric description of the complete linear series}

In the case of complete linear series $g^1_4$, we can describe geometrically the divisors in the image of $\phi$: 
fix two different points $r, s \in X$ such that the line  $l=r+s$ intersects $X$ in four different points. Put $$l\cdot X= r+s+u+v.$$ 

Denote by $t_r, t_s, t_u, t_v$ the tangent lines to $X$ at the points $r, s, u, v$ respectively.  Let us define:
\[
\mathcal F_{r ,s}=\{\text{conics through } u, v \text{ tangent to } t_u, t_v \text{ at } u, v \text{ resp.}\}\cong \mathbb P^1.
\]
If $Q \in \mathcal F_{r,s}$, then $Q\cdot X= 2u+2v+p_1+p_2+p_3+p_4$. All these degree $8$ divisors are linearly equivalent on $X$ 
(they belong to $\vert 2 K_X \vert $ since we are intersecting with a conic). One of these conics is the double line $l^2\in \mathcal F_{r,s}$ which intersects $X$ in the divisor
$2u+2v+2r+2s$. Therefore:
\[
2u+2v+2r+2s\sim 2u+2v+ p_1+p_2+p_3+p_4,
\]
hence $2r+2s\sim p_1+p_2+p_3+p_4$. The description of the $g^1_4$ is now simple: given a point $p_1 \in X$, there is a unique $Q \in \mathcal F_{r,s}$ passing
through $p_1$. Then there is a map $f_{r,s}: X  \lra \mathcal F_{r,s}\cong \mathbb P^1$, sending $p_1$ to this conic. The fibre is the divisor $p_1+p_2+p_3+p_4$ 
considered above. Notice that  one of the fibres is $2r+2s$, hence $f_{r,s}$ is one of the $g^1_4$ we are looking for. 

In the same way taking the pencil $\mathcal F_{u,v}$ of conics tangent to $t_r$ (resp. $t_s$) at $r$ (resp. $s$), intersecting the conics with $X$ and subtracting 
the divisor $2r+2s$ we obtain the linear series $f_{u, v}:X\lra \mathcal F_{u, v}\cong \mathbb P^1$. 

Observe that the involution $i$ sends $f_{r,s}$ to $f_{u,v}$.

\subsection{The curve of $g^1_4$'s with two special fibres}

We need to determine the curve on $\varphi (S)\subset \mathsf G^1_4(X)$ given by the $g^1_4$'s on $X$ with two fibres of type $2p+2q$. In the case of
linear series in $E$ (the non-complete linear series), these clearly correspond to points which are in two bitangents. In the other cases, we have to understand when a  
map $f_{r, s}$ as above has a second fibre of the form $2x+2y$. Thanks to our description now we know that we only have to look at
\[
\begin{aligned}
\Gamma :=\{r+s \in X^{(2)} \mid \exists\, \text{a conic}\, Q \,\text{(of rank at least 2)}  \\
\text{ with } Q\cdot X =2u+2v+2x+2y \}.
\end{aligned}
\]

Consider the composition of maps 
\[
X^{(2)}\times X^{(2)} \xrightarrow{m}  X^{(4)}  \xrightarrow{s} Pic^8(X), 
\]
where $m$ is the addition of divisors and $s$ is the ``square'' map $\sum p_i \mapsto \mathcal O_X(2\sum p_i)$. The map $s$ is surjective since it is 
the composition of two surjective maps: $X^{(4)}\lra Pic^4(X)$ and $Pic^4 (X) \lra Pic^8(X)$, $L\to 2L$.

We observe that $s^{-1}(\omega_X^{\otimes2})$ is the disjoint union of $2^{2\cdot g(X)}=2^6=64$ components. One is isomorphic to a projective plane and it is simply the canonical linear series. This component is rather
uninteresting since it gives only double lines $l^2$. The other $63$ components are projective lines corresponding to the paracanonical systems $\vert \omega_X \otimes \alpha\vert $, $\alpha \in JX_2\smallsetminus \{0\}$. A divisor $D$ in one of these lines is formed by $4$ points 
not in a line and such that there is a conic intersecting $X$ in $2D$. Define $\Gamma_{\alpha }:=m^{-1}(\vert \omega_X \otimes  \alpha \vert) $ for a non trivial
2-torsion point $\alpha $. Then
\[
\Gamma = \bigcup_{\alpha \in JX_2\smallsetminus \{0\}} \Gamma _{\alpha}.
\]

Since $\Gamma $ does not contain points of $\mathcal B$, its preimage in $S$ is isomorphic to $\Gamma$ hence it is a disjoint union of curves in $S$ that we still denote by $\Gamma $. 

Call
$\mathsf U_X$ the open set obtained subtracting to $S$ the set $\Gamma $ and the set of points in the exceptional divisors corresponding to points belonging to 
two bitangents. 

\subsection{The involution on $\mathsf G^1_4(X)$}

We want to prove that the natural involution in $\mathcal R_{3,sp}^{tr}$, which exchanges the trigonal maps $f_x$ and $f_y$, corresponds, via the trigonal construction, to the involution 
\[
i:(X,L) \mapsto (X,\omega_X^{\otimes2}\otimes L\meno) 
\]
in $\mathcal M^{tet}_{3 ,*}$.
\begin{remark}
	The involution in $\mathcal Rb_{3, sp}^{tr}$ does not exchange the covering (it acts only on the trigonal series). Since the Prym variety of the covering is isomorphic to the Jacobian of the associated tetragonal curve we know that the involution $i$ has to leave the curve $X$ invariant.
	
\end{remark}

To prove the equality of the involutions we go in the opposite direction: we fix the quartic $X$ as above and the two complete linear series $f_{r, s}$, $f_{u,v}$ such
that $r, s, u, v$ are on a line. These linear series correspond by the involution $i$. It is enough to prove the coincidence of both involutions for 
these examples since they are the generic elements. \\

Define (following Recillas, see \cite{lange-birkenhake}):
\[
\tilde D_{r,s} =\{ a+b \in X^{(2)} \mid f_{r,s}(a)=f_{r,s}(b) \}.
\]
Notice that there are involutions $\sigma_{r, s}$ (and resp. $\sigma_{u, v}$) on the curves $\tilde D_{r,s}$ (and resp. $ \tilde D_{u,v} $) sending each pair of points to the complement in the corresponding linear series. We denote by $\tilde C_{r, s}$ and $\tilde C_{u, v}$ the quotient (trigonal) curves. Recillas trigonal construction says that there are isomorphisms of principally polarized abelian varieties:
\[
P(\tilde D_{r, s}, \tilde{C}_{r, s}) \cong JX \cong P(\tilde D_{u, v}, \tilde{C}_{u, v}).
\]
The assignment $(X,f_{r, s}) \mapsto (\tilde D_{r, s}, \tilde{C}_{r, s}, M)$ is the inverse of the trigonal construction. $M$ is the $g^1_3$ on $\tilde{C}_{r, s}$ which sends $[p_1+p_2]=[p_3+p_4]$ to the corresponding conic in $\mathcal F_{r, s}$. Its fibre is of the form $\{[p_1+p_2], [p_1+p_3], [p_1+p_4]\}$. Notice that here tetragonal maps $f_{r, s}$ (and resp. $ f_{u, v} $) come with two simple ramification points over $l^2$. This implies that, unlike what occurs in Recillas construction, $(\tilde D_{r, s}, \tilde{C}_{r, s})$ (and resp. $(\tilde D_{u, v}, \tilde{C}_{u, v})$) is an admissible double cover of type $(\ast)$ in the sense of Beauville. Donagi extended the trigonal construction to admissible double covers (for details see \cite[Theorem 2.9]{donagi}). 
Call $(D_{r, s}, {C}_{r, s})$ (and resp. $(D_{u, v}, {C}_{u, v})$) the normalizations of these coverings.

\begin{prop}\label{Propinvoluzione}
	\begin{enumerate}
		\item [a)] There is a canonical isomorphism $\tilde D_{r, s}\xrightarrow{\lambda }  \tilde D_{u, v}$ compatible with the involutions: 
		$\lambda \circ \sigma_{r, s}=\sigma_{u, v} \circ \lambda $. In particular  there is an isomorphism $\tilde{C}_{r, s} \xrightarrow{\overline \lambda} \tilde{C}_{u, v}$. 
		\item [b)] Let $b \in \tilde{C}_{r,s}$ be the branch point of $\tilde  D_{r, s} \xrightarrow{\pi_{r,s}} \tilde{C}_{r,s}$, then $b':=\overline \lambda (b)$ is the branch point of $\tilde D_{u, v}
		\lra \tilde{C}_{u, v}$. 
		\item [c)]There exist points $x \in C_{r,s}$ and $y\in C_{u, v}$ such that $\vert b_1+b_2+x \vert $ and $\vert b_1'+b_2'+y \vert $ are the corresponding
		trigonal series (where $b_i$ and $b_i'$ are the preimages of $b$ and $b'$ in the normalizations of $\tilde{C}_{r,s}$ and $\tilde{C}_{u,v}$).
		\item [d)] There is an isomorphism $\mathcal O_{C_{u, v}}(b_1'+b_2'+\overline \lambda (x)+y)\cong \omega_{C_{u,v}}$.
	\end{enumerate}
\end{prop}
\begin{proof}
	Let $p_1+p_2 \in \tilde D_{r,s}$. By definition $h^0(X,\mathcal O_X(2r+2s-p_1-p_2))=1$. Thus, by Serre duality we have that
	\[
	\begin{aligned}
	1= & h^0(X,\omega_X(p_1+p_2-2r-2s))=\\ &h^0(X,\mathcal O_X(r+s+u+v+p_1+p_2-2r-2s))= \\ &h^0(X,\mathcal O_X(u+v+p_1+p_2-r-s)).
	\end{aligned}
	\]
	Let $q_1+q_2\in \vert u+v+p_1+p_2-r-s \vert$. Let us see that $q_1+q_2 \in \tilde D_{u,v}$. Indeed:
	\[
	\begin{aligned}
	& h^0(X,\mathcal O_X(2u+2v-q_1-q_2))=\\ &h^0(X,\mathcal O_X(2u+2v-u-v-p_1-p_2+r+s))\\ =&h^0(X,\mathcal O_X(u+v+r+s-p_1-p_2))=1.
	\end{aligned}
	\]
	Therefore the map $\tilde D_{r, s}\xrightarrow{\lambda }  \tilde D_{u, v}$ given by   
	\[
	\lambda (p_1+p_2)=q_1+q_2\sim p_1+p_2+u+v-r-s,
	\]
	is well defined and the compatibility with the involutions is an exercise. This proves a). Observe that b) is an obvious consequence once we notice that $\sigma_{r, s}$ has a unique fixed point given by $r+s$. The same occurs in $u+v$ for $\sigma_{u, v}$.  From point a) we know that $\lambda(r+s)=u+v$. Thus, calling $b$ and $b'$ the images of $r+s$ (resp. $u+v$) in $\tilde{C}_{r,s}$ (resp. in $\tilde{C}_{u,v}$), we get $\overline \lambda(b)=b'.$
	
	To prove c) we refer to the description of the extended trigonal construction given by Donagi. Indeed, we have that the fibre of the 3:1 map $\tilde{C}_{r,s}\ra\PP$ over $l^2$ consists of a node in $b$ and an additional point $x=\pi_{r,s}(2r)=\pi_{r,s}(2s)$. The normalization of $\tilde{C}_{r,s}$ gives the trigonal series $|b_1+b_2+x|.$ The same occurs for $\tilde{C}_{u,v}$ calling $y=\pi_{u,v}(2u)=\pi_{u,v}(2v)$.
	
	Finally we conclude with d). First notice that with an abuse of notation we are still calling $\overline{\lambda}$ the isomorphism induced between the normalized curves ${C}_{r,s}\ra{C}_{u,v}$. Then consider $C_{r,s}$ and $C_{u,v}$ as quartic plane curves and the canonical divisors obtained intersecting $C_{r,s}$ (resp. $C_{u,v}$) with the line $b_1+b_2$ (resp. $b_1'+b_2'$). Thus we get $$K_{C_{r,s}}=x+b_1+b_2+z\quad \text{and} \quad K_{C_{u,v}}=y+b'_1+b'_2+w.$$
	Now we have two possibilities: $$w=\overline{\lambda}(x)\quad \text{or} \quad w=\overline{\lambda}(z).$$
	We claim that $ w=\overline{\lambda}(x) $. In fact if $ w=\overline{\lambda}(z) $, since by construction $x=\pi_{r,s}(2r)=\pi_{r,s}(2s)$ then we would have $\lambda(2r)=2u$ or $\lambda(2r)=2v$. But this contradicts the definition of $\lambda$ given above. Hence we get  $\mathcal O_{C_{u, v}}(b_1'+b_2'+\overline \lambda (x)+y)\cong \omega_{C_{u,v}}$.\\
\end{proof}
\begin{remark}
		The isomorphism of Proposition \eqref{Propinvoluzione}[$d)$], gives the compatibility between the two trigonal maps $f_x$ and $f_y$ defined for the general element of $\mathcal{R}b_{3,sp}^{tr}$ and the two trigonal maps obtained on $C_{r,s}$ (resp. $C_{u,v}$) projecting from $x$ or from $z$ (resp. from $\overline{\lambda}(x)$ or from $y$).\\
\end{remark}
\begin{teo}\label{teo5}
	The  fibre of $\mathcal{P}_{3,2}$ at a generic $JX$ is isomorphic to  the quotient of $\varphi (\mathsf U_X)\subset \mathsf G^1_4(X)$ by the involution $i$.
	\begin{proof}
		Starting with a general 3-dimensional abelian variety, i.e. the Jacobian of a curve $X$, diagram \eqref{Diag} says that the fibre of $\mathcal{P}_{3,2}$ over $JX$ is described by the fibre over $X$ of the map $\mathcal{M}_{3,*}^{tet}\ra\mathcal{M}_{3}$. Thus we need to look for all tetragonal maps $k:X\ra\PP$ which have an \'{e}tale point on every fibre and only a fibre with exactly two ramification points of order 2.
		
		 We consider the map $\varphi_0$ in \eqref{blowup} and we look at its image in $Pic^4(X)$. The blow up $S$ of $X^{(2)}$ at $\mathcal{B}$ recovers all tetragonal maps obtained as projections from points on bitangent lines. Hence, considering the open set $\mathsf{U}_X$, we avoid tetragonal maps which have two fibres of type $2p+2q$ (which are not allowed by the trigonal construction). 
		 
		 Finally, since $\mathcal{R}b_{3,sp}^{tr}$ has an involution which exchanges the two special trigonal series, we let $i$ act on $\mathcal{M}_{3,*}^{tet}$ to identify the two tetragonal maps on $X$ which correspond (by the isomorphism \eqref{IsoLangeOrtega}) to the trigonal maps $f_x$ and $f_y$ and we denote by $\mathcal{M}q_{3,*}^{tet}$ the corresponding moduli space. Letting $i$ act on  $\varphi (\mathsf U_X)$, we obtain the fibre over $JX$.\\
	\end{proof}
\end{teo}
\section{Case $g=4,r=2$}

In this last case we identify $\mathcal R_{4,2}$ with $\Delta^{n,0} $, the set of isomorphism classes of irreducible admissible coverings of curves of arithmetic genus $5$ with exactly one node. Notice that $\Delta^{n,0}$ is a dense open set of an irreducible divisor $\Delta^n$ in the boundary of $\overline {\mathcal R}_5$.  

In \cite{donagi} Donagi describes the generic fibre of the extended classical Prym map 
\[
\overline {\mathcal{P}}_{5}: \overline {\mathcal{R}}_5  \ra \mathcal{A}_4.
\]
He defines a birational map 
\[
\kappa: \mathcal A_4 \ra \mathcal {RC}^+,
\]
where $\mathcal {RC}^+$  is the moduli space of pairs $(V,\delta)$,  where $V$ is a smooth cubic threefold $V$ and $\delta \in JV_2$ a non-zero $2$-torsion point in the intermediate Jacobian $JV$ with a "parity" condition. An explicit open set in $\mathcal A_4$ where $\kappa $ is an isomorphism is given in \cite{izadi}. The main theorem in section $5$ of \cite{donagi} says that the fibre of $\kappa \circ \overline {\mathcal P}_5 $ at a generic $(V,\delta)$ is isomorphic to the surface $\widetilde {F(V)}$, which is the unramified double covering of the Fano surface $F(V)$ attached to $\delta $ (remember that $Pic^0(F(V))\cong JV$).   

Our aim is to identify which elements of $\widetilde {F(V)}$ correspond to admissible irreducible double coverings of nodal curves. In other words, we want to find the intersection:
\[
\widetilde {F(V)}\cap \Delta ^{n,0}.
\]

We will prove that the image of this intersection by the double covering 
\[
\tau: \left(\kappa \circ \overline {\mathcal P}_5\right)^{-1}(V,\delta)=\widetilde {F(V)} \ra F(V)
\]
lies in a curve $\Gamma $ already considered in the literature and which is defined as follows:
\[
\Gamma :=\{ l \in F(V) \mid \exists\; \text{a plane}\; \Pi \;\text{and a line}\; r\in F(V)\; \text{with} \;V\cdot \Pi = l+2r\}.
\]

	\begin{remark}\label{open_set}
	 It is stated in \cite[Proposition 2.6]{naranjo-ortega} that the curve $\Gamma $ is smooth for  a generic cubic threefold. There is a mistake in the parameter count of the proof in that paper and in fact  this curve has always nodes. Nevertheless, for a generic element of $\Gamma $ the plane $\Pi $ in the definition above is unique. Hence for a general element $l\in \Gamma$ the discriminant quintic $Q_l$  of the conic bundle structure provided by $l$, has only one node. We denote by $\Gamma_0 \subset \Gamma$ the open set of the points with this property.
	 \end{remark}

Let us denote by $\widetilde \Gamma $ the curve $\tau ^{-1}(\Gamma ).$
Izadi developed in \cite{izadi} the ideas outlined by Donagi in \cite[section 5]{donagi}. In section 3 she studied in detail the action of the involution $\lambda $ associated with $\tau$ and  the intersection of 
$\widetilde{F(V)}$ with the boundary. This intersection is a curve that we call $\Gamma '$ as explained in Proposition \ref{dominant}. 
The curve $\Gamma'$ is interchanged with a curve in the smooth locus of $\widetilde{F(V)}$, the locus of Prym curves with odd vanishing theta null (\cite[p. 121]{izadi}). These two curves map to $\Gamma $. Hence  the preimage of $\Gamma$ breaks into two components. Our aim is to determine the intersection of $\Gamma'$  with $\Delta^{n,0}.$

\begin{prop}
	For a generic cubic threefold $V$ we have that 
	\[
	\tau( \widetilde {F(V)}\cap \Delta ^{n,0})\subset \Gamma.
	\] 
	In particular $\widetilde {F(V)}\cap \Delta ^{n,0} \subset \Gamma'$.   \end{prop}

We have two proofs for this fact. The first follows closely Donagi's description of the fibre and concludes that if $l\in F(V) \smallsetminus \Gamma $ then $\tau ^{-1}(l)$ is given by coverings of smooth curves, that is $\tau^{-1}(l) \subset \mathcal R_5$. 
	
	Let us remind briefly this description: let $A\in \mathcal A_4$ be a generic abelian fourfold and put $\kappa (A)=(V,\delta)$.
	
	Choose a generic line $l\in F(V)$ and denote by $\pi_l:\widetilde {Q_l}\lra Q_l$ the admissible double covering attached to the conic bundle structure on $V$ provided by $l$. Then $Q_l$ is a smooth quintic plane curve and $P(\widetilde {Q_l},Q_l)\cong JV$ (see \cite{beauJacInt} or \cite[Appendix C]{clemgriff}). 
	
		Let $\sigma \in ({JQ_l})_2$ be the $2$-torsion point that determines $\pi_l$. Then, by the general theory of Prym varieties (see \cite[page 332, Corollary 1]{mumford}), there is an exact sequence 
	\begin{equation}\label{two-torsion}
		0 \lra \langle \sigma \rangle \lra \langle \sigma \rangle ^\perp \lra P(\widetilde {Q_l}, Q_l)_2=JV_2 \lra 0, 
	\end{equation}
	where $\langle \sigma \rangle ^\perp \subset ({JQ_l})_2 $ is the orthogonal with respect to the Weil pairing. Denote by $\nu $ a preimage of the fixed $2$-torsion point $\delta $ in $JV_2$, then the other preimage is $\nu':=\nu + \sigma $. Hence $\sigma,\nu,\nu'$define an isotropic subgroup $W_l$ of rank $2$ on $JQ_l$. 
	
	The parity condition of $(V, \delta) \in \mathcal {RC}^+$ means that  $h^0(Q_l,\mathcal O_{Q_l}(1)\otimes \nu )$ and  $h^0(Q_l,\mathcal O_{Q_l}(1)\otimes \nu' )$  are even. Thus there are two curves of genus $5$, $C$ and $C'$, such that $J C\cong P(Q_l, \nu)$ and $JC'\cong P(Q_l, \nu')$. 
	This is due to the existence of a bijection between non-hyperelliptic genus 5 curves $C$ and admissible coverings of quintic plane curves with an even 2-torsion point. The quintic appears as the quotient of  $W^1_4(C)$ by the natural involution.
	% 	
	% 	Looking at Donagi's diagram, these two curves give two lines  that intersect in the point $Q_l$.
	% 		\definecolor{qqqqff}{rgb}{0,0,1}
	% 	\definecolor{uququq}{rgb}{0.25,0.25,0.25}
	% 	\definecolor{tttttt}{rgb}{0.2,0.2,0.2}
	% 	\begin{tikzpicture}[line cap=round,line join=round,>=triangle 45,x=1.0cm,y=1.0cm]
	% 	\clip(2.38,-0.95) rectangle (12.49,6.67);
	% 	\draw [line width=1.6pt] (4,0)-- (11,0);
	% 	\draw [line width=0.8pt,dotted] (11,0)-- (7.5,6.06);
	% 	\draw [line width=0.8pt,dotted] (7.5,6.06)-- (4,0);
	% 	\draw [line width=1.6pt] (7.5,6.06)-- (7.5,0);
	% 	\draw [line width=1.2pt,color=qqqqff] (7.5,2.02) circle (2.02cm);
	% 	\begin{scriptsize}
	% 	\fill [color=black] (4,0) circle (1.5pt);
	% 	\fill [color=tttttt] (11,0) circle (1.5pt);
	% 	\draw[color=black] (8.73,-0.18) node {$\tilde C$};
	% 	\fill [color=uququq] (7.5,6.06) circle (1.5pt);
	% 	\fill [color=uququq] (7.5,0) circle (1.5pt);
	% 	\draw[color=uququq] (7.58,-0.24) node {$Q_l$};
	% 	\draw[color=black] (7.11,2.4) node {$\tilde C'$};
	% 	\draw[color=qqqqff] (6.53,3.51) node {$V$};
	% 	\end{scriptsize}
	% 	\end{tikzpicture}
	
	Using for $P(Q_l,\nu)$ an exact sequence similar to (\ref{two-torsion}), we get: 
	\[
	0 \lra \langle \nu \rangle \lra \langle \nu \rangle ^\perp \lra P(Q_l,\nu)_2=J{C}_2 \lra 0.
	\]
	Therefore the rank 2 subgroup $W_l\subset \langle \nu \rangle ^\perp $ determines on $JC$ a $2$-torsion point $\mu$. Similarly there is a $\mu ' \in J C'_2$. Denoting with $\lambda $ the sheet interchange for the covering $\tau$, Donagi proves that:
	\[
	\lambda (C,\mu)=( C',\mu') \qquad \text {and } \qquad  P( C,\mu) \cong P(C,\mu')\cong A. 
	\]
	Hence the preimages of $l$ by $\tau $ are the elements $( C,\mu), (C',\mu')$ obtained previously. In particular, since they are smooth, we find that $\tau^{-1}(l) \subset \mathcal R_5$, as claimed. This concludes the first proof. \\
	
	The second proof is more constructive and more useful for our purposes. We show directly that for a covering in $\Delta ^{n,0}$ the corresponding line $l$ belongs to $\Gamma $. This approach relies on the following result of Izadi (see \cite[Theorem 6.13]{izadi}):
	
	\begin{teo}\label{izadi}
		Let $(V,\delta)$ be  a generic smooth cubic threefold endowed with a
		non-zero 2-torsion point and let $\pi^*: D^* \ra C^*$ be an admissible covering in the fibre of $(V,\delta)$. Assume that $\tau (  \pi^*)=l\in F(V)$. Then the discriminant quintic $Q_l$ of the conic bundle structure attached to $l$ parametrizes the set of singular quadrics through the canonical model of $ C^*$.
	\end{teo}
	By canonical model we mean the image of $C^*$ by the morphism attached to the dualizing sheaf.
	\begin{remark} 
		The line $l$ attached to $\pi^*$ is defined in \cite{izadi} in a different way. However it is proved in 6.30 in loc. cit. that it equals $\tau (\pi^*)$.  
	\end{remark}
	Let $(V,\delta)$ be  a generic smooth cubic threefold endowed with a
		non-zero 2-torsion point and let $\pi : D\ra C$ be the generic element in $\mathcal R_{4,2}$ of the fibre of $\mathcal P_{4,2}$ above  $(V,\delta)$. We denote by  $\pi^* : D^*\ra C^*$ the corresponding admissible covering in $\Delta ^{n,0} \subset \overline{ \mathcal{R}}_5$. By definition $C^*=C/b_1\sim b_2$ is a curve of arithmetic genus $5$ with a node in $p$ obtained by glueing the two branch points $b_1, b_2$ of $\pi$.
	\begin{lemma}\label{Lemmaquintiche}Under the above assumptions:\begin{itemize}
			\item[a)] the quintic plane curve parametrizing the singular quadrics containing the image of the canonical map of $C^*$ is a quintic with exactly one node. In particular 
			 $\tau (\pi ^*) \in \Gamma_0.$
			\item[b)] The quintic plane curve parametrizing the singular quadrics containing the canonical image of an arithmetic genus 5 curve with at least two nodes is a nodal quintic with at least two nodes.
		\end{itemize}
		\end{lemma}
	\begin{proof}
     Since the general fibre of $\mathcal P_{4,2}$ has pure dimension 1, by dimensional reason, it will not intersect the locus where $C$ is hyperelliptic and it will intersect the locus where $C$ admits a unique $g^1_3$ in at most a finite number of points. Therefore it is enough to prove part $ a) $ assuming $C$ not hyperelliptic and with two distinct $g^1_3$'s.
     
	The map $\varphi: C \ra \mathbb P(H^0(C, \omega_C(b_1+b_2))^*)$ satisfies $\varphi (b_1)=\varphi (b_2)$ and it is an isomorphism out of these two points. Hence $\varphi (C)=C^*$ and $\varphi$ can be seen as the normalization  $n:C\ra C^*$ composed with the inclusion $C^*\subset \mathbb{P}(H^0(C, \omega_C(b_1+b_2)))^*=\mathbb{P} ^4$.
	
	We have the following exact sequence:
	\begin{equation}\label{normalization}
			0\ra \omega_{C^*}\ra n_*(\omega_C(b_1+b_2))\ra \mathbb{C}_p\ra0,    
	\end{equation}
	which induces 
\begin{equation}\label{coomnormalization}
		0\ra H^0(C^*,\omega_{C^*})\ra H^0(C,\omega_C(b_1+b_2))\xrightarrow{res}\mathbb{C}\ra\mathbb{C}\ra0,
\end{equation}
 where $res$ is the map $\omega\mapsto \text{res}_{b_1}\omega+\text{res}_{b_2}\omega$. By the residue theorem it vanishes identically. Therefore 
	\[ 
	H^0(C^*,\omega_{C^*})\cong H^0(C,\omega_C(b_1+b_2)).  
	\]
Now let $L$ be a $g^1_3$ on $C$ and consider bases 
	\[
	H^0(C,L)=\langle t_1,t_2\rangle, \qquad H^0(C,\omega_C\otimes L^{-1})=\langle s_1,s_2\rangle.
	\]
	Put 
	\[
	\omega_1=t_1s_1 \quad  \omega_2=t_2s_1 \quad  \omega_3=t_1s_2 \quad  \omega_4=t_2s_2,    
	\]
	to get 
	
$$ 	H^0(C,\omega_C)=\langle \omega_1,\omega_2, \omega_3, \omega_4\rangle $$ and then completing the basis we get $ H^0(C,\omega_C(b_1+b_2))=\langle \omega_1,\omega_2, \omega_3, \omega_4, \omega_5\rangle $.

	We obtain the following diagram:
	\begin{equation}\label{cono}
		\begin{tikzcd}
			C\arrow[r,"\varphi"]\arrow[dr,hook, "g"]&\mathbb{P}^4\arrow[d,dashed]\\
			&\mathbb{P}^3
		\end{tikzcd}
	\end{equation} 
	where $g$ is the canonical map and the vertical rational map is given by dualizing the inclusion $H^0(\omega_C)\subset H^0(\omega_C(b_1+b_2))$. It corresponds to the projection from the point $p$.
	
	Since $C$ is a general curve of genus 4, there exists a unique quadric  $Q$ containing its canonical model and it has rank $4$, namely (here $\odot$ denotes the symmetric product): $Q=\omega_1\odot\omega_4-\omega_2\odot\omega_3.$ In particular, in the chosen coordinates, $\varphi(b_i)=p=[0:0:0:0:1]$ ($i=1,2$) and  $Q=\{x_1x_4-x_2x_3=0\}.$ 
	The preimage of $Q$ by the projection is a cone with vertex  $p$ which contains $C^*$ and has rank four (and in fact the same equation). We still call it $Q$. 
	
	Using now (see e.g. \cite[pp. 90]{ACGH2})
	\begin{equation}\label{normalizationtwisted}
		0\ra \omega_{C^*}^{\otimes2}\ra n_*(\omega_C^{\otimes2}(2b_1+2b_2))\ra \mathbb{C}_p\ra0
	\end{equation}
	and its corresponding long exact sequence in cohomology, we obtain that also in the case of a nodal curve of arithmetic genus 5 \[\dim \ker (Sym^2H^0(C^*, \omega_{C^*})\ra H^0(C^*, \omega_{C^*}^{\otimes2}))=\dim I_2(\omega_{C^*})=3.\]
	Taking $ Q, Q_1,Q_2$ as a basis, we would like to show that the discriminant curve $\Delta$ of the family of quadrics 
	\[
	\mathbb{P}(I_2(\om_{C^*}))=\mathbb P(\langle Q, Q_1,Q_2\rangle)\]
	is nodal. 
	By the above considerations $p\in S(Q)\cap C^*$, where $ S(\cdot) $ denotes the singular locus.
	
	In the paper \cite{wall}, Wall studied the discriminant locus of nets of quadrics. In particular	\cite[Lemma 1.1]{wall} ensures that $([1:0:0], p)$ belongs to $S(N)$, where 
	\[
	N:=\{([\lambda_0:\lambda_1:\lambda_2],x)\mid x^t(\lambda_0 A_0+\lambda_1 A_1+\lambda_2 A_2)x=0 \}\subset \mathbb P(\langle Q, Q_1,Q_2\rangle)\times \mathbb P^4
	\]
	is the universal family of the net of quadrics containing $C^*$ ($A_i, i=0,1,2$ are the matrices associated to $Q,Q_1,Q_2$). To be more precise we would have to write $Q=Q_{[1:0:0]}$ and the analogous for other $Q_i$. We will omit the subscript when it will be possible.
	
	Assuming that every point in $S(C^*)$ is tame (we give the definition below), the map $S(N)\ra S(C^*)$, which sends $(\lambda=[\lambda_0:\lambda_1:\lambda_2],x)$ to $ x$, becomes bijective. Since, in our case, $S(C^*)=\{p\}$, we obtain that $S(N)=\{([1:0:0],p)\}$. Moreover $p\in S(C^*)$ and $([1:0:0],p) \in S(N)$ have the same type of singularity (by \cite[Proposition 1.3]{wall}).
	\begin{claim}\label{mapproj}
		The map \begin{align*}
		\rho: N&\ra\mathbb{P}^2\\
		(\lambda, x)&\mapsto\lambda \notag
		\end{align*} sends $S(N)$ to $S(\Delta)$.
		\begin{proof}
			First observe that since\[\rho\meno(\Delta)=\{(\lambda,x) : x \in Q_{\lambda}\;\text{and}\; Q_{\lambda}\;\text{is singular}\},\]
			we get \[S(N)\subseteq \rho\meno(\Delta).\] 
			We conclude with the Jacobian criterion. 
			Indeed, using local coordinates for which a singular point of $N$ is $ 	([1:0:0],[0:0:0:0:1]) $, we have: 
			\begin{equation*}
			Q\leftrightarrow\begin{pmatrix}
			a_1 &0 &0 &0 &0\\
			0 &a_2 &0 &0 &0\\
			0 &0 &a_3 &0 &0\\
			0 &0 &0 &a_4 &0\\
			0 &0 &0 &0 &0\\
			\end{pmatrix}, \; Q_1\leftrightarrow\begin{pmatrix}
			& &  & & \\
			& &  & &\\
			& &\ast  & &\\
			& &  & &\\
			& &  & &0\\
			\end{pmatrix},
			\; Q_2\leftrightarrow\begin{pmatrix}
			& &  & & \\
			& &  & &\\
			& &\ast  & &\\
			& &  & &\\
			& &  & &0\\
			\end{pmatrix},
			\end{equation*}
			since $ p $ is a point in all quadrics of the net and $ Q $ is singular in $p$. 
			Therefore 	$ \lambda_0A_0+\lambda_1A_1+\lambda_2A_2 $ has a $ 0 $ in position $(5,5)$, homogeneous linear polynomials $l=l(\lambda_1,\lambda_2)$ out of the diagonal and $l+\lambda_0a_i$ on the diagonal ($i=1,2,3,4$).
			
			Put $ G(\lambda_0,\lambda_1,\lambda_2)=det(\lambda_0A_0+\lambda_1A_1+\lambda_2A_2 ) $. Then it is possible to write 
			$$G=f_5+\lambda_0f_4+\lambda_0^2f_3+\lambda_0^3f_2,$$
			where  $ f_i $ are homogeneous polynomials in $(\lambda_1,\lambda_2)$ of degree $i$. 
			Therefore $\partial_\lambda G(p)=0$, i.e. $G$ is singular at $p$ and thus $[1:0:0]$ belongs to $S(\Delta)$.\\
		\end{proof}
	\end{claim}
	
	 Applying \cite[Theorem 1.4]{wall} for $([1:0:0])$ in $S(\Delta)$ (and resp. $([1:0:0],p) \in S(N)$), we conclude that the discriminant locus of $N$ has a unique nodal point, as claimed.
	 
	It only remains to show that $p$ is tame. By definition a point of $C^*$ is \textit{tame} if the tangent planes to $C^*$ at the point span a $2$-dimensional vector space.
	We check that $p$ is tame: call $\pi_i, i=0,1,2$ the tangent planes of the three quadrics at $p$.
	In coordinates: 
	\[
	\pi_i: (0:0:0:0:1)A_i \textbf{y}=0,\qquad i=0,1,2.
	\] 
	Thus: 
	\[
	p A_0 \textbf{y}=(0:0:0:0:1) \begin{pmatrix}
	0 &0 &0 &1 &0\\
	0 &0 &-1 &0 &0\\
	0 &-1 &0 &0 &0\\
	1 &0 &0 &0 &0\\
	0 &0 &0 &0 &0\\
	\end{pmatrix} \textbf{y}=0
	\]
	and 
	\[
	p A_i \textbf{y}=0 \Leftrightarrow (a_{5,1}^i,a_{5,2}^i,a_{5,3}^i,a_{5,4}^i,a_{5,5}^i) \textbf{y}=0, 
	\]
	where $a^i_{5,j}$ are the coefficients of the last row of the matrices $A_i, i=1,2$. Call these vectors $\mathbf{a_1},\mathbf{a_2}$. Then $p$ is tame if 
	\[
	\dim \langle \mathbf{a_1},\mathbf{a_2}\rangle=2. 
	\]
	Suppose, by contradiction, that $\mathbf{a_2}=\mu\mathbf{a_1}$. Then $Q': \mu A_1-A_2$ belongs to $I_2(K_C)$, so $Q'=\nu Q$. Therefore 
	\[
	0=\nu Q -Q'=\nu Q-\mu Q_1+Q_2,
	\]
	which is impossible. This concludes the proof of part $a)$.
	
	In order to prove part $b)$, let us start with an admissible double-nodal covering $\pi^{2*}: D^{2*}\ra C^{2*}$, i.e. $S(C^{2*})=\{p_1,p_2\}$. Consider the following partial normalization maps:
	\begin{equation*}
	\begin{tikzcd}
	N_1\arrow[r,"\alpha"]\arrow[rr,bend left, "n"]&\tilde{N}_1\arrow[r,"\beta"]&C^{2*}
	\end{tikzcd}
	\end{equation*}
	where $n$ is the normalization, $\beta$ is the partial normalization of the node in $p_2$ while $\alpha$ of the one in $p_1$. By dimension reason we can assume that $N_1$ is not hyperelliptic.
	
	A short exact sequence for $\omega_{\tilde{N}_1}$ similar to \eqref{normalization} ensures that $$\dim I_2(\omega_{\tilde{N}_1})=\dim I_2(\omega_{N_1}(q_1+q_1'))=1,$$ where $q_1,q_1'$ are the two points of $N_1$ sent to $p_1$ by $\alpha$. 
	
	We remark that the unique quadric $Q$ containing the image of  \[N_1\ra\tilde{N}_1\subset \mathbb{P}(H^0(\omega_{N_1}(q_1+q_1'))) \] cannot be singular in $p_1$. Otherwise, if we write in local coordinates $$Q=\sum_{i,j\leq 4}a_{ij}x_ix_j\quad  \text{and} \quad  p_1=[0:0:0:1],$$ we would get $\partial_iQ(p_1)=a_{i4}=0$ for every $i$. But this would imply $Q\in I_2(\omega_{N_1})$, which is impossible since $I_2(\omega_{N_1})=0$. Then, dualizing the inclusion $H^0(\omega_{\tilde{N}_1})\subset H^0(\omega_{\tilde{N}_1}(q_2+q_2'))$ (points $q_2,q_2'$ are identified in $p_2$ in $C^{2*}$), we obtain a diagram as \eqref{cono} where the rational map $\mathbb{P}^4\dashrightarrow\mathbb{P}^3$ is given by the projection from $p_2$. The preimage of $Q$ is a cone with vertex $p_2$ which is smooth in $p_1$ and which contains our curve with two nodes. With an abuse of notation, we still denote it by $Q$. 
	
	The short exact sequence \eqref{normalizationtwisted} for the bicanonical $\omega_{C^{2*}}^{\otimes2}$ of $C^{2*}$ shows that $\dim I_2(\omega_{C^{2*}})=3.$ Call, as before, $N\subset \mathbb{P}(I_2(\omega_{C^{2*}}))\times \mathbb{P}^4$ the universal family of the net of quadrics containing $C^{2*}$. Thus, $Q$ is the point $\lambda=[1:0:0]$ in $\mathbb{P}^2$. 
	
	Since $C^{2*}$ has two singular points, \cite[Lemma 1.2]{wall} shows that there exists $\mu\in \Delta$ (the discriminant curve of the net $\mathbb{P}(I_2(\omega_{C^{2*}}))$) such that $p_1\in S(Q_{\mu})$. Hence $Q\neq Q_{\mu}$. 
	This concludes the proof: $(\lambda, p_2) $ and $(\mu, p_1) $ belong to $S(N)$. Therefore the map $\rho$ of Claim \ref{mapproj} (which works also in case of $Q$ of rank 3, that is $a_4=0$) determines two different singular points in $\Delta$.
	
	The cases $\#S(C^*)=3,4$ are similar: the partial normalization at one point leads to a curve of arithmetic genus 4 with singular points. Call one of them $p_1$. As above we find a quadric $Q$ in $I_2$ which is a cone with vertex $p_1$ on a quadric which is smooth in at least one of the remaining nodes. Applying Wall's theorems, we know the existence of another quadric which is singular in at least one among the other nodes. This leads to a discriminant curve $\Delta$ which has at least two singular points. \\
\end{proof}
Thus the following holds  (see \ref{open_set} for the definition of $\Gamma_0$):
\begin{teo}\label{teo6}
	The generic fibre of $\mathcal P_{4,2}$  at $(V,\delta)$ is isomorphic to $ \Gamma_0$.
	\begin{proof}
		Take $\cov$ in $\mathcal R_{4,2}$ and denote, as above, $\pi^*$ the corresponding element in $\Delta^{n,0}$. Lemma \ref{Lemmaquintiche}$a)$ shows that $\tau(\pi^*)$ belongs to $\Gamma_0 $. Therefore, in order to show that the generic fibre of $\mathcal P_{4,2}$  at $(V,\delta)$ is isomorphic to $\Gamma_0$,  it remains to prove that an element in $\Gamma '$ with two or more nodes maps to $\Gamma \setminus \Gamma_0$. Since $A\in \mathcal{A}_4$ is generic, we can suppose $A$ simple and hence, using \cite{beau}, we can just take into account coverings of irreducible curves. Finally, the inclusion $\Delta^{n,0}\subset \Delta^n$ guarantees that we only have to take care of admissible coverings of irreducible curves with more than one node. Therefore, suppose by contradiction that $\widetilde \Gamma$ contains an admissible covering of an irreducible curve with (at least) two nodes. Lemma \ref{Lemmaquintiche}$b)$ gives us a quintic plane curve with at least two nodes. This contradicts the assumption on  $\Gamma_0$ and thus we can conclude. \\
	\end{proof}
\end{teo}

\section{Fibres of the Prym map and Shimura varieties}
In this section we give some examples of irreducible components of some fibres of the ramified Prym maps which yield Shimura subvarieties of ${\mathcal A}_g$. 
Recall that in \cite{moonen-special}, \cite{moonen-oort}, \cite{fgp}, \cite{fpp} examples of Shimura subvarieties of ${\mathcal A}_g$ generically contained in the Torelli locus have been constructed  as families of Jacobians of  Galois covers of ${\mathbb P}^1$ or of elliptic curves. Some of them are contained  in fibres of ramified Prym maps. In \cite{fgs} and \cite{gm} infinitely many examples of totally geodesic and of Shimura varieties generically contained in the Torelli locus have been constructed as fibres of ramified Prym maps.

In particular, the images in ${\mathcal M}_2$ and in ${\mathcal M}_3$ of ${\mathcal R}_{1,2}$, respectively ${\mathcal R}_{1,4}$, are  the bielliptic loci and in \cite{fpp} it is shown that their images in ${\mathcal A}_2$, resp. ${\mathcal A}_3$, via the Torelli maps are Shimura subvarieties.
In \cite{fgs} it is proven  that the irreducible components of the fibres of the Prym maps ${\mathcal P}_{1,2}$, ${\mathcal P}_{1,4}$ are totally geodesic curves and countably many of them are Shimura curves. 
Moreover in \cite{fgs} the authors show that family $(7) = (23) = (34)$ of \cite{fgp} is a fibre of the Prym map ${\mathcal P}_{1,4}$, which is a Shimura curve. 

It is easy to see that  the Shimura family (24) of \cite{fgp} is contained in a fibre of the Prym map ${\mathcal P}_{2,2}$. In fact it is a family of curves $D$ of genus 4 with an action of a group $G = {\mathbb Z}/2 \times  {\mathbb Z}/2 \times {\mathbb Z}/3$, such that  $D/G \cong {\mathbb P}^1$ and the map $D \to D/G \cong {\mathbb P}^1 $ is branched over  $B$ which consists of 4 distinct points. In terms of the generators $g_1, g_2, g_3$ of $G$, with $o(g_1) = o(g_2) = 2$, $o(g_3) = 3$, the monodromy of the covering $\theta: \pi_1({\mathbb P^1} \setminus B)  \cong \langle \gamma_1,...,\gamma_4 \ | \ \gamma_1 \gamma_2 \gamma_3 \gamma_4 = 1 \rangle \to G$ is $\theta(\gamma_1) =  g_2$, $\theta(\gamma_2) =  g_1g_2$, $\theta(\gamma_3) =  g_3 $, $\theta(\gamma_4) =  g_1g_3^2$. 

One easily checks that the map $D \to D/\langle g_1 \rangle$ is a double covering of a genus 2 curve, ramified over 2 points.  Moreover the Prym variety $P(D,C)$ is isogenous to $E \times E'$ where $E = D/\langle g_2 \rangle$ and $E' = D/\langle g_1g_2 \rangle$ and $E$ and $E'$ do not move, since the Galois covers $E \to E/(G/\langle g_2 \rangle) = D/G = {\mathbb P}^1$ and $E' \to E'/(G/\langle g_1g_2 \rangle) = D/G = {\mathbb P}^1$ both have only 3 critical values. This shows that the family of covers $D \to C$ is contained in a fibre of the Prym map ${\mathcal P}_{2,2}$. 

Finally, we give an explicit  new example of a totally geodesic curve which is an irreducible component of a fibre of the Prym map ${\mathcal P}_{1,2}$. \\

{\bf Example.} 

Consider a family of Galois covers $\psi_{\lambda} : D_{\lambda} \to D_{\lambda}/G \cong {\mathbb P}^1$,  ramified over $B = \{P_1 = \lambda, P_2=1, P_3=0, P_4 = \infty\}$ with $G \cong (\mathbb Z/4 \times \mathbb Z/4) \ltimes \mathbb Z/2$ and with $g(D_{\lambda} ) =11$.  We use the following presentation of $G$:
\begin{gather*}
	G \cong \langle g_1, g_2, g_3, g_4, g_5  \ | \ g_1^8=g_2^2=g_3^4=g_4^4=g_5^2=1,\ g_1^2 = g_4, \\ g_3^2 = g_4^2 = g_5, \ g_1^{-1} g_2 g_1 = g_2g_3, \  g_1^{-1} g_3 g_1 = g_3g_5, \ g_2^{-1} g_3 g_2 = g_3g_5 
	\rangle.
\end{gather*}
Notice that $G = (\langle g_1g_2g_3 \rangle \times \langle g_4 \rangle) \ltimes \langle g_2 \rangle \cong (\mathbb Z/4 \times \mathbb Z/4) \ltimes \mathbb Z/2$. For simplicity, as above,  we omit the index $\lambda$ and we denote an element of the family of Galois covers simply by $\psi: D \to D\to D/G \cong  {\mathbb P}^1$. The monodromy of the cover $\theta: \pi_1({\mathbb P^1} \setminus B)  \cong \langle \gamma_1,...,\gamma_4 \ | \ \gamma_1 \gamma_2 \gamma_3 \gamma_4 = 1 \rangle \to G$ is
\[[\theta(\gamma_1) =  g_2g_3g_5,  \ \theta(\gamma_2) =  g_3g_4g_5, \ \theta(\gamma_3) =  g_1g_2g_4g_5, \ \theta(\gamma_4) =  g_1g_3g_4g_5]\] and these elements have  orders $[2,2,4,8]$ in $G$. 

Consider the subgroup $H = \langle g_2, g_5\rangle \cong \mathbb Z/2 \times \mathbb Z/2$ of $G$. By Riemann Hurwitz formula one easily  computes the genus of the quotient $D/H$ which is 2.  Set  $K : =  \langle g_2, g_3g_4 \rangle \cong D_4$, $K_1 :=  \langle g_2, g_4 \rangle \cong \mathbb Z/2 \times  \mathbb Z/4 $, $K_2 :=  \langle g_2, g_3 \rangle \cong D_4$. The genus 2 curve $C$ admits three distinct double covers:  
\begin{gather*}
	f: C = D/H \to D/ K  \cong {\mathbb P}^1,\\
	f_1: C = D/H \to D/ K_1=:E_1, \  f_2: C = D/H \to D/ K_2 =: E_2,
\end{gather*}
where $E_1$ and $E_2$ are elliptic curves. 

The double covers $\pi_1: E_1 = D/ K_1 \to D/ \langle  g_2, g_3, g_4\rangle \cong {\mathbb P}^1$,  $\pi_2: E_2 = D/ K_2 \to D/ \langle  g_2, g_3, g_4\rangle \cong {\mathbb P}^1$ allow to express the elliptic curves $E_1$ and $E_2$ in Legendre form: 
$$E_1: y^2 = x(x-\mu)(x^2 -1), \ \ E_2: y^2 = x(x^2-1),$$
where $\mu^2 = \lambda$. 

Therefore the elliptic curve $E_2$ does not move and $J(C)$ is isogenous to $E_1 \times E_2$. So the Prym varieties $P(C,E_1)$ are isogenous to the fixed elliptic curve $E_2$. Thus the 1-dimensional family of double covers $ \pi_1: C \to E_1$ is contained in a fibre of the Prym map ${\mathcal P}_{1,2}$, hence it gives an irreducible component of a fibre of ${\mathcal P}_{1,2}$. 
 

\begin{thebibliography} {99}
	
	
\bibitem{ACGH} 
E.~Arbarello, M.~Cornalba, P.~A. Griffiths, and 	  J.~Harris.  {\em Geometry of algebraic curves. {V}ol. {I}}, volume 267 of {\em Grundlehren der
	Mathematischen Wissenschaften}.  \newblock Springer-Verlag, (1985).
\bibitem{ACGH2} 
E.~Arbarello, M.~Cornalba and P.~A. Griffiths.  {\em Geometry of algebraic curves. {V}ol. {II}}, volume 268 of {\em Grundlehren der
	Mathematischen Wissenschaften}.  \newblock Springer-Verlag, (2011).

\bibitem{Barth} 
W. Barth, {\em Abelian Surfaces with (1,2)-Polarization},    Adv. Stud. Pure Math. 10 (1987), 41-84.

\bibitem{beau} 
A. Beauville, {\em Prym varieties and the Schottky problem}, Inventiones Math. 41 (1977), 149-96.

\bibitem{beauJacInt} 
A. Beauville, {\em Vari\'{e}t\'{e}s the Prym et jacobiennes interm\'{e}diaires}, Ann. Sci. \'{E}c. Norm. Sup\'{e}r 4 (1977), 309-391.

\bibitem{bcv}
F.~Bardelli, C.~Ciliberto and A.~Verra, {\em Curves of minimal genus on a general abelian variety},  Compositio Math. 96 (1995), 115-147. 


\bibitem{birkenhake-langepol} C. Birkenhake; H. Lange.  {\em An isomorphism between moduli spaces of abelian varieties}, Math. Nachr. 253 (2003), 3–7.


\bibitem{lange-birkenhake} 
C.~Birkenhake and H.~Lange.  {\em Complex abelian varieties}, volume 302 of Grundlehren der Mathematischen Wissenschaften. Springer-Verlag, Berlin, second edition, (2004).

\bibitem{clemgriff}  C. H. Clemens and P. A. Griffiths.  {\em The Intermediate Jacobian of the Cubic Threefold}. Annals of Math. 95(2), pp. 281-356, (1972).

\bibitem{donagi}
R.~Donagi, {\em The fibers of the Prym map}, Curves, Jacobians, and abelian varieties (Amherst, MA, 1990), Contemp. Math. 136 (1992), Amer. Math. Soc.,  55-125. 

\bibitem{ds}
R.~Donagi, R.~Smith, {\em The structure of the Prym map}, Acta Math. 146 (1981), 25-102.

\bibitem{Farkas} 
G.~Farkas, {\em Prym varieties and their moduli}, in  Contributions to algebraic geometry (2012), EMS, Z\"urich, 215-255.

\bibitem{fgp} P.~Frediani, A.~Ghigi and M.~Penegini. {\em Shimura
  varieties in the Torelli locus via Galois coverings}, 
    Int. Math. Res. Not. 20 (2015), 10595-10623.


\bibitem{fgs}
P.~Frediani, A.~Ghigi, I.~Spelta, {\em Infinitely many Shimura varieties in the Jacobian locus for $g \leq 4$}, arXiv:1910.13245, to appear in Annali SNS.

\bibitem{fpp} P.~Frediani, M.~Penegini, P.~Porru, {\em   Shimura
  varieties in the Torelli locus via Galois coverings of elliptic
  curves},  Geometriae Dedicata 181 (2016), 177-192.

\bibitem{gm} S.~Grushevsky, M.~M\"oller, {	\em Explicit
  formulas for infinitely many Shimura curves in genus 4},
   Asian J. Math. 22 (2018), no. 2, 381-390.
  
\bibitem{ikeda}
A. Ikeda, {\em Global Prym-Torelli Theorem for double coverings of elliptic curves}, Algebr. Geom. 7 (2020), 544-560.

\bibitem{izadi}
E.~Izadi, {\em The geometric structure of $\mathcal A_4$ , the structure of the Prym map, double solids and  $\Gamma_{00}$-divisors}, J. Reine Angew. Math. 462 (1995), 93-158.

\bibitem{lange-ortega} 
H.~Lange, A.~Ortega {\em The Trigonal construction in the ramified case}, preprint, arxiv:1902.00251.

\bibitem{moonen-special} B.~Moonen, {\em Special subvarieties
  arising from families of cyclic covers of the projective line},
   Doc. Math. 15 (2010), 793-819.
 
\bibitem{moonen-oort} B.~Moonen, F.~Oort,{\em   The {T}orelli
  locus and special subvarieties}, {H}andbook of
    {M}{oduli} (2013), no. 2, 549-594. 


\bibitem{mumford} 
D.~Mumford, {\em Prym varieties, I.}, Contributions to Analysis (a collection of papers dedicated to Lipman Bers), 325-350, Academic Press, New York, 1974. 

\bibitem{mn}
V.~Marcucci, J.C.~Naranjo, {\em Prym varieties of double coverings of elliptic curves},  Int. Math. Res. Notices 6 (2014), 1689-1698.

\bibitem{mp}
V.~Marcucci, G.P.~Pirola, {\em Generic Torelli for Prym varieties of ramified coverings}, Compositio Mathematica 148 (2012), 1147-1170.

\bibitem{nr}
D.S.~Nagaraj, S. Ramanan, {\em Polarisations of type $(1,2,\dots,2)$ on abelian varieties}, Duke Math. J.  80 (1995), 157--194.  

\bibitem{naranjo-ortega}
J.C.~Naranjo, A.~Ortega, {\em Generic injectivity of the Prym map for double ramified coverings},  Trans. Amer. Math. Soc. 371 (2019), 3627-3646.

\bibitem{naranjo-ortega2}
J.C.~Naranjo, A.~Ortega, {\em Global Prym-Torelli for double coverings ramified in at least $6$ points}, to appear in J. Algebraic Geom. 

\bibitem{Pantazis} 
S.~Pantazis, {\em Prym varieties and the geodesic flow on SO(n)}, Math. Annalen 273 (1986), 297-315.

%\bibitem{Reci} S. Recillas, {\em PhD Thesis }, Brandeis Univ., (1974).

\bibitem{ReciTrig} 
S.~Recillas, {\em Jacobians of curves with $g^1_4$ are Prym varieties of trigonal curves}, Bol. Soc. Mat. Mexicana 19 (1974), 9-13.

\bibitem{verra} 
A.~Verra, {\em The Fibre of the Prym Map in Genus Three}, Math. Annalen 276 (1987), 433-448.

\bibitem{wall} 
C.T.C.~Wall, {\em Singularities of nets of quadrics}, Compositio Mathematica  42 (1980), 187-212.




%%%%%%%%%%%%%%%%%%%%%%%%%%%%%%%%%%%%%%%%%%%%%%%%%%%%%%%%%%%%%%%%%%%%%%%%%%%%%%%%%%%%%%%%%%%%%%%%%%%%%%%%%%%%%%%%%%%%%



% 
% \bibitem{acg2} E.~Arbarello, M.~Cornalba, and P.~A. Griffiths.
%   \newblock {\em Geometry of algebraic curves. {V}ol. {II}},
%   volume 268 of {\em Grundlehren der Mathematischen
%   Wissenschaften}.  \newblock Springer-Verlag, New York, 2011.
% 
% \bibitem{beau77} A. Beauville, {\em{Vari\'et\'es de Prym et
%   jacobiennes interm\'ediaires}}. (French) Ann. Sci. \'Ecole
%   Norm. Sup. (4) 10 (1977), no. 3, 309-391.


% \bibitem{cf2} E.~Colombo and P.~Frediani.  \newblock Some results
%   on the second {G}aussian map for curves.  \newblock {\em
%   Michigan Math. J.}, 58(3):745--758, 2009.
% 
% \bibitem{cf1} E.~Colombo and P.~Frediani.  \newblock Siegel metric
%   and curvature of the moduli space of curves.  \newblock {\em
%   Trans. Amer. Math. Soc.}, 362(3):1231--1246, 2010.
% 
% \bibitem{cfprym}Colombo, ~E., Frediani, ~P., Prym map and second
%   Gaussian map for Prym-canonical line bundles.  \newblock {\em
%   Adv. Math.} 239 (2013), 47-71.
% 
%	\bibitem{cfg} E.~Colombo, P.~Frediani, and A.~Ghigi.  \newblock On
%	totally geodesic submanifolds in the {J}acobian locus.
%	\newblock{\em International Journal of Mathematics}, 26 (2015),
%	no. 1, 1550005 (21 pages).

% 
% 
% 
% \bibitem{ds}R.~Donagi, R. ~Smith, \newblock The structure of the
%   Prym map. \newblock {\em Acta Math.} 146 (1981), no. 1-2,
%   25--102.
% 
% 
% \bibitem{Farkas} G. Farkas; \newblock Prym varieties and their
%   moduli. In \newblock {\em Contributions to algebraic geometry},
%   215-255, EMS, Z\"urich, 2012.
% 
% \bibitem{fl} G. Farkas; K. Ludwig, The Kodaira dimension of the
%   moduli space of Prym varieties. {\em J. Eur. Math. Soc.}) 12
%   (2010), no. 3, 755-795.
% 


% \bibitem{friedman-smith} Friedman, Robert; Smith, Roy, The generic
%   Torelli theorem for the Prym map. {\em Invent. Math.} 67 (1982),
%   no. 3, 473-490.
% 
% \bibitem{GAP4} The GAP~Group, \emph{GAP -- Groups, Algorithms, and
%   Programming, Version 4.8.7}; 2017,
%   \verb+(http://www.gap-system.org)+.
% 
% \bibitem{gavino} G.~Gonz{\'a}lez~D{\'{\i}}ez and W.~J. Harvey.
%   \newblock Moduli of {R}iemann surfaces with symmetry.  \newblock
%   In {\em Discrete groups and geometry ({B}irmingham, 1991)},
%   volume 173 of {\em London Math. Soc. Lecture Note Ser.}, pages
%   75--93. Cambridge Univ. Press, Cambridge, 1992.
%



% \bibitem{kanev-global-Torelli} V.~I. Kanev.  \newblock A global
%   {T}orelli theorem for {P}rym varieties at a general point.
%   \newblock {\em Izv. Akad. Nauk SSSR Ser. Mat.}, 46(2):244--268,
%   431, 1982.
%


% \bibitem{lange-ortega} H.~Lange and A.~Ortega.  \newblock Prym
%   varieties of cyclic coverings.  \newblock {\em Geom. Dedicata},
%   150:391--403, 2011.
% 
% \bibitem{liu-yau} K.~Liu, X.~Sun, X.~Yang, and S.-T. Yau.
%   \newblock Curvatures of moduli spaces of curves and
%   applications.  \newblock {\tt arXiv:1312.6932 [math.DG]}, 2013.
%   \newblock Preprint.
% 
% \bibitem{lz} X.~Lu and K.~Zuo.  \newblock The {O}ort conjecture
%   for on {S}himura curves in the {T}orelli locus of curves.
%   \newblock {\em arXiv preprint arXiv:1405.4751}, 2014.
% 

%	\bibitem{mohajer-zuo-paa} A.~Mohajer and K.~Zuo.  \newblock On Shimura
%	subvarieties generated by families of abelian covers of
%	$\mathbb{P}^1$.  {\em J. Pure Appl. Algebra} 222 (2018), no. 4,
%	931--949.

%	\bibitem{moonen-linearity-1} B.~Moonen.  \newblock Linearity
%	properties of {S}himura varieties. {I}.  \newblock {\em J. Algebraic
%		Geom.}, 7(3):539--567, 1998.

%	\bibitem{moonen-special} B.~Moonen.  \newblock Special subvarieties
%	arising from families of cyclic covers of the projective line.
%	\newblock {\em Doc. Math.}, 15:793--819, 2010.

%	\bibitem{moonen-oort} B.~Moonen and F.~Oort.  \newblock The {T}orelli
%	locus and special subvarieties.  \newblock In {\em {H}andbook of
%		{M}{oduli: Volume II}}, pages 549--94.  International {P}ress,
%	Boston, MA, 2013.

%	\bibitem{mumford-Shimura} D.~Mumford.  \newblock A note of {S}himura's
%	paper ``{D}iscontinuous groups and abelian varieties''.  \newblock
%	{\em Math. Ann.}, 181:345--351, 1969.
% 
%

%	\bibitem{naeff} R.~Naeff, {\em The Chevalley-Weil formula}, Master
%	Thesis,University of Amsterdam, 2005.

%	\bibitem{oort-can} F.~Oort.  \newblock Canonical liftings and dense
%	sets of {CM}-points.  \newblock In {\em Arithmetic geometry
%		({C}ortona, 1994)}, Sympos. Math., XXXVII, pages
%	228--234. Cambridge Univ. Press, Cambridge, 1997.
% 
% \bibitem{pauro} J.~Paulhus and A.~Rojas.  \newblock Completely
%   decomposable Jacobian varieties in new genera.  {\em
%   Experimental Mathematics}, 2016, 1--16.
% 
% \bibitem{matteo2011} M.~Penegini.  \newblock The classification of
%   isotrivially fibred surfaces with {$p_g=q=2$}.  \newblock {\em
%   Collect. Math.}, 62(3):239--274, 2011.  \newblock With an
%   appendix by S{\"o}nke Rollenske.
% 
% \bibitem{penegini2013surfaces} M.~Penegini.  \newblock Surfaces
%   isogenous to a product of curves, braid groups and mapping class
%   groups.  \newblock In {\em Beauville surfaces and groups},
%   volume 123 of {\em Springer Proc. Math. Stat.}, pages
%   129--148. Springer, Cham, 2015.

%	\bibitem{pirola-Xiao} G.~P. Pirola.  \newblock On a conjecture of
%	{X}iao.  \newblock {\em J. Reine Angew. Math.}, 431:75--89, 1992.

% \bibitem{rohde} J.~C. Rohde.  \newblock {\em Cyclic coverings,
%   {C}alabi-{Y}au manifolds and complex multiplication}, volume
%   1975 of {\em Lecture Notes in Mathematics}.  \newblock
%   Springer-Verlag, Berlin, 2009.
%

%	\bibitem{rojas} A.~M. Rojas, {\em Group actions on Jacobian
%		varieties}, Rev. Mat. Iber. {\bf 23} (2007), 397--420.


%	\bibitem{serre-echelon} J.~P. Serre. \newblock Appendix to \newblock
%	Grothendieck, {\em Techniques de construction en gomtrie
%		analytique. X.} \newblock Sminaire Henri Cartan, Volume 13
%	(1960-1961) no. 2, Talk no. 17.

% \bibitem{S} J.-P. Serre.  \newblock {\em Repr\'esentations
%   lin\'eaires des groupes finis}.  \newblock Hermann, Paris,
%   revised edition, 1978.
% 
% \bibitem{wi} W. Wirtinger, \newblock {\em Untersuchungen \"uber
%   Thetafunktionen}, Teubner, Berlin 1895.
\end{thebibliography}
\end{document}